\documentclass[a4paper,11pt]{article}

\usepackage{amsfonts,qtree,stmaryrd,textcomp}
\usepackage{pifont,amssymb,amsmath,amsthm,tabularx,graphicx}
\usepackage{tree-dvips,pictexwd,dcpic}
\usepackage{bbm}
\usepackage{bm}

\usepackage{stmaryrd}

\usepackage[T1]{fontenc}
\usepackage[latin1]{inputenc}
\usepackage[text={17cm,26cm},centering]{geometry}

\usepackage{etex}
\usepackage{fancyhdr}
\usepackage{graphicx}
\usepackage{enumitem}
\usepackage{amsmath}
\usepackage[bbgreekl]{mathbbol}% order sensitive; it is before amssymb
\usepackage{amssymb}
\usepackage{amsthm}
\usepackage[all]{xy}

\usepackage{epigraph}
\RequirePackage{bussproofs}

\usepackage{mathabx}

\usepackage{scalerel,stackengine}

\usepackage{framed}
\usepackage{mathtools}
\usepackage{url}
\usepackage{proof}
\usepackage{xspace}
\usepackage{listings}
\usepackage{wrapfig}
\usepackage{tikz}
\usepackage{tikz-cd}

\usepackage{relsize}

\usetikzlibrary{positioning}
\usetikzlibrary{shapes}

\usetikzlibrary{arrows}
\usetikzlibrary{calc}
\usepackage{tikz-3dplot}
\usetikzlibrary{decorations.pathmorphing}

\usepackage{amsfonts,amssymb,amsmath,qtree,stmaryrd,textcomp, amsthm}

\usetikzlibrary{decorations.pathreplacing}
\tikzset{axis/.style={&lt;-&gt;}}

\stackMath
\newcommand\reallywidehat[1]{%
\savestack{\tmpbox}{\stretchto{%
  \scaleto{%
    \scalerel*[\widthof{\ensuremath{#1}}]{\kern-.6pt\bigwedge\kern-.6pt}%
    {\rule[-\textheight/2]{1ex}{\textheight}}%WIDTH-LIMITED BIG WEDGE
  }{\textheight}% 
}{0.5ex}}%
\stackon[1pt]{#1}{\tmpbox}%
}
\usepackage{cancel}

 \definecolor{MyBlue}{rgb}{0.05, 0.25, 0.65}
 
 \definecolor{MyRed}{rgb}{0.90, 0.05, 0.05}
\definecolor{MyGreen}{rgb}{0.05, 0.90, 0.05}

\newcommand{\B}{\boldsymbol}
\newcommand{\C}[1]{\mathcal{#1}}
\newcommand{\D}[1]{\mathbb{#1}}
\usepackage{ mathrsfs }

\newtheorem{theorem}{Theorem}[section]
\newtheorem{proposition}[theorem]{Proposition}

\newtheorem{corollary}[theorem]{Corollary}
\newtheorem{remark}[theorem]{Remark}

\newtheorem{example}[theorem]{Example}

\newtheorem{definition}[theorem]{Definition}

\newcommand{\Nat}{{\mathbb N}}
\newcommand{\Real}{{\mathbb R}}

\newcommand{\id}{\mathrm{id}}

\newcommand{\BISH}{\mathrm{BISH}}

\newcommand{\CST}{\mathrm{CST}}

\newcommand{\dom}{\mathrm{dom}}

\newcommand{\PEM}{\mathrm{PEM}}

\newcommand{\INT}{\mathrm{INT}}

\newcommand{\TOT}{\Leftrightarrow}

\newcommand{\To}{\Rightarrow}

\newcommand{\sto}{\rightsquigarrow}

\newcommand{\MLTT}{\mathrm{MLTT}} 
\newcommand{\CZF}{\mathrm{CZF}}

\newcommand{\prb}{\textnormal{\textbf{pr}}}

\newcommand{\pr}{\textnormal{\texttt{pr}}}

\newcommand{\BST}{\mathrm{BST}}

\newcommand{\Disj}{\mathbin{\B ) \B (}}

\newcommand{\Set}{\mathrm{\mathbf{Set}}}

\newcommand{\Ineq}{\textnormal{\texttt{Ineq}}}

\newcommand{\SetIneq}{\textnormal{\textbf{SetIneq}}}

\newcommand{\SetExtIneq}{\textnormal{\textbf{SetExtIneq}}}

\newcommand{\MIN}{\mathrm{MIN}}

\newcommand{\EFQ}{\mathrm{EFQ}}

\newcommand{\tii}{\mathrm{II}}

\newcommand{\emptys}{\cancel{\mathlarger{\mathlarger{\mathlarger{\mathlarger{\square}}}}}}

\newcommand{\LNP}{\mathrm{LNP}}

\newcommand{\CLNP}{\mathrm{CLNP}}

\newcommand{\PWF}{\mathrm{WF}} 
\newcommand{\IND}{\mathrm{IND}}

\newcommand{\Prime}{\textnormal{\texttt{Prime}}}

\newcommand{\Coprime}{\textnormal{\texttt{Coprime}}}

\newcommand{\Peano}{\textnormal{\texttt{Peano}}}

\newcommand{\Bisho}{\textnormal{\texttt{Bishop}}}

\newcommand{\wfs}{\textnormal{\texttt{wfs}}}

\newcommand{\Fun}{\textnormal{\textbf{Fun}}}

\newcommand{\StrExtFun}{\textnormal{\textbf{StrExtFun}}}

\newcommand{\SetExtIneqRel}{\textnormal{\textbf{SetExtIneqRel}}}

\newcommand{\StrExtFunRel}{\textnormal{\textbf{StrExtFunRel}}}

\newcommand{\WFSet}{\textnormal{\textbf{WFSet}}}

\newcommand{\inj}{\textnormal{\texttt{inj}}}

\newcommand{\lex}{\textnormal{\texttt{lex}}}

\begin{document}

%\fonttable{MnSymbolA10}

\date{}

%\title{\textbf{A constructive least number principle}}

\title{\textbf{Coinductive well-foundedness}}

%\title{\textbf{$\exists$-Well-Founded Sets}}

%\title{\textbf{The Ex falso Principle in Constructive Mathematics}}
%\titlecomment{{\lsuper*}This paper is a major extension of~\cite{Pe17}.}

\author{Iosif Petrakis\\	%required
Department of Computer Science, University of Verona\\
iosif.petrakis@univr.it}  %optional
%\thanks{thanks 1, optional.}	%optional

% \author[B.~Name2]{Bob Name2}	%optional
% \address{address2; addresses should initially be duplicated, even if
%   authors share an affiliation}	%optional
% \email{name2@email2; ditto for email addresses}  %optional
% \thanks{thanks 2, optional.}	%optional
% 
% \author[C.~Name3]{Carla Name3}	%optional
% \address{address 3}	%optional
% \urladdr{name3@url3\quad\rm{(optionally, a web-page can be specified)}}  %optional
% \thanks{thanks 3, optional.}	%optional

%% etc.

%% required for running head on odd and even pages, use suitable
%% abbreviations in case of long titles and many authors:

%%%%%%%%%%%%%%%%%%%%%%%%%%%%%%%%%%%%%%%%%%%%%%%%%%%%%%%%%%%%%%%%%%%%%%%%%%%

%% the abstract has to PRECEDE the command \maketitle:
%% be sure not to issue the \maketitle command twice!

\maketitle

\begin{abstract}
	\noindent 
We introduce a coinductive version $\exists \PWF_{\Nat}$ of the well-foundedness of $\Nat$ that is used in our proof within minimal logic of the constructive counterpart $\CLNP$ to the 
standard least number principle $\LNP$. According to $\CLNP$, an inhabited complemented subset of $\Nat$ has a least element if and only if it is downset located. The use of complemented subsets of $\Nat$ in the formulation of $\CLNP$, instead of subsets of $\Nat$, allows a positive approach to the subject that avoids negation. Generealising $\exists \PWF_{\Nat}$, we define $\exists$-well-founded sets and we prove their fundamental properties.
	\end{abstract}


\section{Introduction}
\label{sec: intro}

We introduce a coinductive version $\exists \PWF_{\Nat}$ of the well-foundedness of $\Nat$. This is only classically equivalent to strong induction, namely it is implied constructively by strong induction (Proposition~\ref{prp: equiv}(i)), but we have only shown classically that it implies strong induction (Proposition~\ref{prp: equiv}(ii).
%and it has to be added in the constructive study of $\Nat$, if we want to work within a negationless constructive framework. 
The principle $\exists \PWF_{\Nat}$ formulates a simple algorithm, described in section~\ref{sec: PWF}. Using $\exists \PWF_{\Nat}$, we prove the divisibility of a natural number by a prime number, avoiding the classical least number principle $(\LNP)$. According to the latter, a non-empty subset $A$ of $\Nat$ has a (unique) least element $\min_A$ i.e., $\min_A \in A$ and $\forall_{x \in \Nat}(x \in A \To \min_A \leq x)$. $\LNP$ cannot be accepted constructively, since it implies a form of the principle of excluded middle $(\PEM_{\Nat})$ over a constructive and very weak set-theoretic framework (Proposition~\ref{prp: LNPtoPEM}). By constructive mathematics we mean Bishop's informal system of constructive mathematics $\BISH$ (see~\cite{Bi67, BB85, BR87}). The theory of sets underlying $\BISH$ was sketched in~\cite{Bi67, BB85} and elaborated in~\cite{Pe20}.  
  In this note we formulate positively i.e., avoiding (weak\footnote{On the difference between weak and strong negation in constructive mathematics we refer to~\cite{KP25}.}) negation, a constructive version $\CLNP$ of the least number principle. According to it, an inhabited complemented subset of $\Nat$ has a least element if and only if it is downset located. Complemented subsets were introduced by Bishop in~\cite{Bi67}, in order to capture complementation in measure theory in a positive way. Complemented subsets were also used in Bishop-Cheng measure theory that was introduced in~\cite{BC72} and extended seriously in~\cite{BB85}. In~\cite{PW22, MWP24, MWP25} the abstract structure of a \textit{swap algebra} induced by the algebra of complemented subsets of a set is studied.  
 The principle $\CLNP$ is classically equivalent to $\LNP$, and its proof is within minimal logic. 
 Our work on the constructive study of $\LNP$ is in analogy to the constructive study of the least upper bound principle, or of the greatest lower bound principle for real numbers by Bishop and Bridges in~\cite{BB85}. The standard least upper bound principle implies the same form of the principle of excluded middle (actually using the same subset $A_P$, see~\cite{BV06}, p.~32), while in~\cite{BB85}, p.~37,  it is shown constructively that an inhabited subset $A$ of $\Real$ has a least upper (greatest lower) bound if and only if it is bounded above (below) and it is upper (lower) order located.

 We structure this paper as follows: 
 \vspace{-2mm}
 \begin{itemize}
 	
 \item In section~\ref{sec: nat} we include the basic properties of the equality and inequality on $\Nat$. 
 
 \item In section~\ref{sec: PWF} we introduce $\exists \PWF_{\Nat}$, a constructive and coinductive version of the well-foundedness of $\Nat$, which is classically equivalent to the non-existence of an infinite descending sequence in $\Nat$. 
 %We use  $\PWF_{\Nat}$ in the proof of $\CLNP$.
 
 \item In section~\ref{sec: cs} we present some basic facts on complemented subsets of $\Nat$ related to $\CLNP$.
 
 \item In section~\ref{sec: CLNP} we introduce downset located complemented subsets of $\Nat$, and we prove $\CLNP$ within minimal logic and using $\exists \PWF_{\Nat}$ (Theorem~\ref{thm: CLNP}).
 
  \item In section~\ref{sec: wfs}, and generalising $\exists \PWF_{\Nat}$, we introduce $\exists$-well-founded-sets and we prove some of their fundamental properties. 
   
 \end{itemize}
 \vspace{-2mm}
 We work within Bishop Set Theory $(\BST)$, a semi-formal system for $\BISH$ that behaves as a high-level programming language. For all notions and results of $\BST$ that are used here without definition or proof we refer to~\cite{Pe20, Pe22a, Pe24}. The type-theoretic interpretation of Bishop's set theory into the theory of setoids (see especially the work of Palmgren~\cite{Pa12}--\cite{Pa19}) has become nowadays the standard way to understand Bishop sets. Other formal systems for BISH are Myhill's Constructive Set Theory $(\CST)$, introduced  in~\cite{My75}, and Aczel's system $\CZF$ (see~\cite{AR10}).
 For all notions and results from Bishop's theory of sets that are used here without explanation or proof, we refer to~\cite{Pe20, Pe22a, Pe24}.  For all notions and results from constructive analysis within $\BISH$ that are used here without explanation or proof, we refer to~\cite{Bi67, BB85, BR87, BV06}. 
 % For all notions and results from Bishop's theory of sets that are used here without explanation or proof, we refer to~\cite{Pe20, Pe22a, Pe24}.

 \section{Natural numbers within $\BISH$}
 \label{sec: nat}

 Within $\BISH$ the natural numbers $\Nat$ is a primitive set equipped with a primitive \textit{equality} and a primitive \textit{inequality} $x \neq_{\Nat} y$. If $1 := S(0)$, where $S \colon \Nat \to \Nat$ is the primitive successor function, the top and the bottom in $\Nat$ are the following formulas, respectively
 $$\top_{\Nat} := 0 \neq_{\Nat} 1, \ \ \ \ \ \bot_{\Nat} := 0 =_{\Nat} 1.$$
 The following axioms on $\Nat$ are accepted: \\[1mm]
 $(\Peano_1)$ $\top_{\Nat}$ .\\[1mm]
 $(\Peano_2)$ $S$ is an \textit{embedding} i.e., $\forall_{x, y \in \Nat}\big(S(x) =_{\Nat} S(y) \To x =_{\Nat} y\big)$.\\[1mm]
 $(\Peano_3)$ or $\IND_{\Nat}$:  $\big[P(0) \ \& \ \forall_{x \in \Nat}(P(x) \To P(S(x)))\big] \To \forall_{x \in \Nat}P(x),$ where $P(x)$ is an \textit{extensional} formula on $\Nat$, that is\footnote{Extensional formulas on a set in $\BISH$ incorporate ``by definition'' the \textit{transport} of Martin-L\"of Type Theory $(\MLTT)$ (see~\cite{HoTT13}).} $\forall_{x, y \in \Nat}\big([x =_{\Nat } y \ \& \ P(x)] \To P(y)\big)$.\\[1mm]
 $(\Bisho_1)$ The equality $x =_{\Nat} y$ is an equivalence relation.\\[1mm]
 $(\Bisho_2)$ The inequality $x \neq_{\Nat} y$ is a \textit{decidable} $(\Ineq_4)$ \textit{apartness relation} $(\Ineq_1 - \Ineq_3)$, where\\[1mm]
 $(\Ineq_1) \  \forall_{x, y \in \Nat}\big(x =_{\Nat} y \ \& \ x \neq_{\Nat} y \To \bot_{\Nat} \big)$.\\[1mm]
 $(\Ineq_2)  \ \forall_{x, y \in \Nat}\big(x \neq_{\Nat} y \To y \neq_{\Nat} x\big)$.\\[1mm]
 $(\Ineq_3) \ \forall_{x, y \in \Nat}\big(x \neq_{\Nat} y \To \forall_{z \in \Nat}(z \neq_{\Nat} x  \vee  z \neq_{\Nat} y)\big)$.\\[1mm]
 $(\Ineq_4) \ \forall_{x, y \in \Nat}\big(x =_{\Nat} y \vee x \neq_{\Nat} y\big)$.\\[2mm]
 The extensionality $(\Ineq_5)$ of $\neq_{\Nat}$ follows from $(\Ineq_1 - \Ineq_3)$, see~\cite{Pe20} Remark 2.2.6, where\\[1mm]
 $(\Ineq_5) \ \forall_{x, x{'}, y, y{'} \in \Nat}\big(x =_{\Nat} x{'} \ \& \ 
 	y =_{\Nat} y{'} \ \& \ x \neq_{\Nat} y \To x{'} \neq_{\Nat} y{'}\big)$.\\[2mm]
The \textit{prime formulas} in $\BISH$ are of the form: $s =_{\Nat} t$, $s \neq_{\Nat} t$, where $s, t \in \Nat$. The \textit{complex formulas} in $\BISH$ are defined as follows: if $A, B$ are formulas, then $A \vee B, A \wedge B, A \To B$ are formulas, and if $S$ is a set and $\phi(x)$ is a  formula, for every variable $x$ of set $S$, then $\exists_{x \in S} \big( \phi(x)\big)$ and $\forall_{x \in S} \big(\phi(x)\big)$ are formulas. If $P$ is a formula in $\BISH$, the \textit{weak negation} $\neg_{\Nat} P$ of $P$ is the formula $\neg_{\Nat} P := P \To \bot_{\Nat}$.
By $(\Ineq_1)$ we get $x \neq_{\Nat} y \To \neg_{\Nat} \big(x =_{\Nat} y)$ i.e., the strong inequality $x \neq_{\Nat} y$ implies the weak inequality $\neg_{\Nat} \big(x =_{\Nat} y)$. The converse implication also follows constructively for $\Nat$ (Proposition~\ref{prp: botnat}(v)).

 All subsets $A$ of $\Nat$ considered here are \textit{extensional} i.e., $A := \{x \in \Nat \mid P(x)\}$, where $P(x)$ is an extensional formula on $\Nat$. The totality of extensional subsets of $\Nat$ is denoted by $\C E(\Nat)$, and the equality on $\C E(\Nat)$ is defined in the obvious way.
 Let $\PEM_{\Nat}$ be the axiom scheme $P \vee \neg_{\Nat} P$, where $P$ is any formula in $\BISH$, such that 
 there is at most one variable $x$ of $\Nat$ ocurring in $P$ and in that case $P(x)$ is extensional. We need this hypothesis of extensionality, in order to apply the separation scheme, which concerns bounded (i.e., formulas with bounded quantifiers only) extensional formulas.
 We include the following standard proof, in order to stress that it is within a certain version of minimal logic $(\MIN)$, which is intuitionisitic logic $(\INT)$ without the ex falso principle $\EFQ_{\Nat}$: $\bot_{\Nat} \To Q$, where $Q$ is an arbitrary formula in $\BISH$. 

\begin{proposition}[$\MIN$]\label{prp:  LNPtoPEM} 
$\LNP$ implies $\PEM_{\Nat}$.
\end{proposition}
	
\begin{proof}
If $P$ is a formula as it is indicated in $\PEM_{\Nat}$, let $A_P := \{x \in \Nat \mid x =_{\Nat} 1\} \cup \{x \in \Nat \mid x =_{\Nat} 0 \ \& \ P\}$. Clearly, $\min_{A_P} = 0$, or $\min_{A_P} = 1$. If $\min_{A_P} = 0$, then $0 \in A_P$ and hence $P$. If $\min_{A_P} = 1$, we suppose $P$. In this case $0 \in A_P$, hence $0 = \min_{A_P}$. By the uniqueness of $\min_{A_P}$ we get $0 =_{\Nat} 1$ i.e., $\neg_{\Nat} P$. 
\end{proof}

Many instances of $\EFQ_{\Nat}$ are provable within $\MIN$. Next follow some of them.

\begin{proposition}[$\MIN$]\label{prp: botnat} The following hold:\\[1mm]
	\normalfont (i)
	\itshape $\bot_{\Nat} \To \forall_{x \in \Nat}\big(x \neq_{\Nat}
	x\big)$.\\[1mm]
	\normalfont (ii)
	\itshape $\bot_{\Nat} \To \forall_{x \in \Nat}\big(0 =_{\Nat}
	x \big)$.\\[1mm]
	\normalfont (iii)
	\itshape $\bot_{\Nat} \To \forall_{x, y \in \Nat}\big(x =_{\Nat}
	y\big)$.\\[1mm]
	\normalfont (iv)
	\itshape $\bot_{\Nat} \To \forall_{x, y \in \Nat}\big(x \neq_{\Nat}
	y\big)$.\\[1mm]
	\normalfont (v)
	\itshape $\forall_{x, y \in \Nat}\big(\neg_{\Nat} \big(x =_{\Nat} y) \To x \neq_{\Nat} y\big)$.\\[1mm]
		\normalfont (vi)
	\itshape $\forall_{x, y \in \Nat}\big(\neg_{\Nat} (x \neq_{\Nat} y) \To x =_{\Nat} y\big)$
\end{proposition}

\begin{proof}
	(i) By induction on $x \in \Nat$. First we show that $0 \neq_{\Nat} 0$. As $0 \neq_{\Nat} 1$, and since by hypothesis $0 =_{\Nat} 1$, by the extensionality of $\neq_{\Nat}$ we get $0 \neq_{\Nat} 0$. The implication $x \neq_{\Nat} x \To S(x) \neq_{\Nat} S(x)$ follows by the injectivity of the successor function $S \colon \Nat \to \Nat$ i.e., $x \neq_{\Nat} y \To S(x) \neq_{\Nat}  S(y)$, for every $x, y \in \Nat$. To show the latter, we suppose that $S(x) =_{\Nat} S(y)$ and since $S$ is an embedding we get $x =_{\Nat} y$. As $x \neq_{\Nat} y$ by hypothesis, we get $\bot_{\Nat}$  from $(\Ineq_1)$. Using the rule $\big[(P \vee Q) \ \& \ \neg_{\Nat} P\big] \To Q$ and the decidability of $\neq_{\Nat}$ we get $S(x) \neq_{\Nat} S(y)$.\\
	(ii) Again we use induction on $x \in \Nat$. The base case $0 =_{\Nat} 0$ follows immediately. The implication 
	$0 =_{\Nat}	x \To 0 =_{\Nat} S(x)$ follows from the fact that if $0 =_{\Nat} x$, then $S(0) =_{\Nat} S(x)$ i.e., $1 =_{\Nat} S(x)$. By the definition of $\bot_{\Nat}$ and the equivalence relation properties of $=_{\Nat}$ we get the required $0 =_{\Nat} S(x)$.\\
	(iii) It follows immediately by case (ii) and the fundamental properties of $=_{\Nat}$.\\
	(iv) It follows immediately by cases (i), (iii) and the extensionality of $\neq_{\Nat}$.\\
	(v) By decidability $x =_{\Nat} y \vee x \neq_{\Nat} y$. If $\neg(x =_{\Nat} y)$ by hypothesis, the rule  $\big[(P \vee Q) \ \& \ \neg_{\Nat} P\big] \To Q$ again gives us $x \neq_{\Nat} y$.\\
	(vi) We work exactly as in the proof of (v).
\end{proof}

\begin{remark}\label{rem: ontheproof}
	\normalfont
	If we use the properties of multiplication on $\Nat$, cases (i) and (ii) of the previous proposition can be shown without induction: if $0 =_{\Nat} 1$, then $0 =_{\Nat}  0 \cdot x =_{\Nat} 1 \cdot x =_{\Nat} x$, for every $x \in \Nat$, and hence $x =_{\Nat} y$, for every $x, y \in \Nat$. By the extensionality of $\neq_{\Nat}$ the inequality $0 \neq_{\Nat} 1$, together with the equalities $x =_{\Nat} 0$ and $1 =_{\Nat} x$ imply that $x \neq_{\Nat} x$, for every $x \in \Nat$. 
\end{remark}

The canonical orders $<_{\Nat}, \leq_{\Nat}$ on $\Nat$ can be defined using the operation of addition, or of cut-off subtraction (see~\cite{TvD88}, p.~124). They can also be introduced as primitive extensional relations on $\Nat$, and they satisfy all expected properties. For example, for every $x, y \in \Nat$ we have that\footnote{For simplicity we omit the subscripts from $<_{\Nat}, \leq_{\Nat}$, and we write $<$ and $\leq$, respectively.}: \\[2mm]
%\noindent
(I1) $x < y \vee x \geq y$,\\[1mm]
(I2) $x \geq y \TOT x = y \vee x > y$.\\[1mm]
(I3) $x < y \To x \neq_{\Nat} y$.\\[1mm]
(I4) $x \neq_{\Nat} y \To x < y \vee y < x$.\\[1mm]
(I5) $\neg_{\Nat}((x < y) \To x \geq y$.

\begin{definition}[Categories of sets]\label{def: sets} Let $(\Set, \Fun)$ be the category\footnote{We denote a category by the pair of its objects \textit{and} arrows.}  of sets and functions, and let $(\SetIneq, \StrExtFun)$ be the category of sets with an inequality that satisfies the corresponding properties $(\Ineq_1, \Ineq_2)$ for $X$, of and strongly extensional functions i.e., functions $f \colon X \to Y$ satisfying $f(x) \neq_Y f(x{'}) \To x \neq_X x{'}$, for every $x, x{'} \in X$. Let $(\SetExtIneq, \StrExtFun)$ be the category of sets with an extenional inequality i.e., the corresponding property $(\Ineq_5)$ holds for $X$, and of strongly extensional functions.	
\end{definition}
	
Clearly, 
%we have the following sequence of subcategories
$(\SetExtIneq, \StrExtFun) \leq (\SetIneq, \StrExtFun) \leq (\Set, \Fun)$, and $(\Nat, =_{\Nat}, \neq_{\Nat}) \in \SetExtIneq$. 
%The arrows in the category of sets with an inequality are \textit{strongly extensional functions} i.e., functions $f \colon X \to Y$ satisfying $f(x) \neq_Y f(x{'}) \To x \neq_X x{'}$, for every $x, x{'} \in X$. 
All functions of type $\Nat \to \Nat$ are strongly extensional. This we cannot accept constructively for every function of type $\Real \to \Real$, as this is equivalent to Markov's principle (see~\cite{Di20}, p.~40). We call a function $f \colon \Nat \to \Nat$ \textit{strongly monotone}, if $\forall_{x, y \in X}\big(f(x) < f(y) \To x < y\big)$. The proof of Proposition~\ref{prp: se} is trivial, and it is based on $(\Ineq_4)$.

\begin{proposition}\label{prp: se}
	Let $f \colon \Nat \to \Nat$ be a function.\\[1mm]
	\normalfont (i)
	\itshape $f$ is strongly extensional.\\[1mm]
	\normalfont (ii)
	\itshape If $f$ is monotone, then $f$ is strongly monotone.\\[1mm]
	\normalfont (iii)
	\itshape If $f$ is strongly monotone and an embedding, then $f$ is monotone.
\end{proposition}

\section{The coinductive principle $\exists \PWF_{\Nat}$ of well-foundedness of $\Nat$}
\label{sec: PWF}

The main principle regarding $<_{\Nat}$ that will be used in the proof of Theorem~\ref{thm: CLNP} is the following constructive and coinductive version of well-foundedness of $\Nat$:\\[2mm]$(\exists \PWF_{\Nat})$ $\exists$-well-foundedness of $\Nat$: if $P(x), Q(x)$ are formulas on $\Nat$ that respect $=_{\Nat}$, then
$$\bigg[\exists_{x \in \Nat}Q(x) \ \& \ \forall_{x \in \Nat}\bigg(Q(x) \To \big[P(x) \vee \exists_{y \in \Nat}\big(y < x \ \& \ Q(y)\big)\big]\bigg)\bigg] \To \exists_{x \in \Nat}P(x).$$
%such that \\[1mm]
%$(\PWF_1)$ $\exists_{x \in \Nat}Q(x),$ and\\[1mm]
%$(\PWF_2)$ $\forall_{x \in \Nat}\bigg(Q(x) \To P(x) \vee \exists_{y \in \Nat}\big(y < x \ \& \ Q(y)\big)\bigg)$.\\[1mm]
%Then $\exists_{x \in \Nat}P(x)$.
$\exists \PWF_{\Nat}$ is a sort of dual to the standard induction principle $\IND_{\Nat}$; its conclusion is an existential formula, and the second disjunct in its hypothesis involves a backward step, rather than a forward one. The algorithm of finding $x \in \Nat$ with $P(x)$ that is captured by this principle is the following: if $x_0 \in \Nat$ with $Q(x_0)$, then either $P(x_0)$, or there is $x_1 < x_0$ with $Q(x_1)$. In the second case, we repeat the argument, and either $P(x_1)$, or there is $x_2 < x_1$ with $Q(x_2)$. After at most $(x_0 +1)$-number of steps we findd $x \in \Nat$ with $P(x)$, since if $Q(0)$ is the case, there is no $x < 0$, and hence $P(0)$ must hold.
	
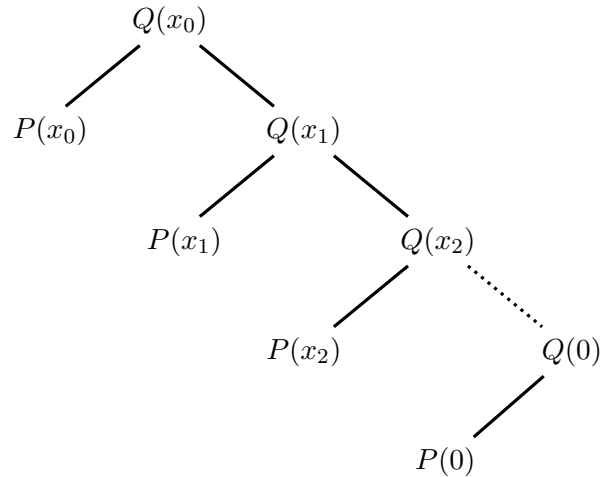
\begin{center}
	\begin{tikzpicture}
		
		\node (A) at (0,0) {$Q(x_0)$};
		\node[right=of A] (B) {};
		\node[left=of A] (C) {};
		\node[below=of B] (D) {$Q(x_1)$};
		\node[below=of C] (E) {$\ \ \ P(x_0)$};
		%\node[right=of A] (B) {$L(\Real^n)$};
		\node[right=of D] (K) {};
		\node[below=of A] (L) {};
		\node[below=of K] (M) {$  Q(x_2)$};
		\node[below=of L] (N) {$\ \ \ P(x_1)$};
		
		\node[right=of M] (S) {};
		\node[below=of S] (T) {$Q(0)$};
		
		\node[left=of M] (W) {};
		\node[below=of W] (Z) {$P(x_2)$};
		
		\node[left=of T] (U) {};
		\node[below=of U] (X) {$P(0)$};

		\draw[-, very thick] (A)--(D) node [midway,above] {};
		\draw[-, very thick] (A)--(E)node [midway,below] {} ;
		
		\draw[-, very thick] (D)--(M) node [midway,above] {};
		\draw[-, very thick] (D)--(N)node [midway,below] {} ;
		\draw[, dotted, very thick] (M)--(T)node [midway,below] {} ;
		
		\draw[-, very thick] (M)--(Z)node [midway,below] {} ;
		
		\draw[-, very thick] (T)--(X)node [midway,below] {} ;
		
	\end{tikzpicture}
\end{center}
Such an anrgument can be used in a classical proof of $\LNP$: if $x_0 \in A$, then classically either $x_0 =_{\Nat} \min_A(x_0)$, or there is $x_1 \in A$ with $x_1 < x_0$, and so on. I.e, $Q_A(x) :\TOT x \in A$ and $P_A(x) :\TOT \forall_{y \in A}(y \geq x)$. 
So far, we know that $\exists \PWF_{\Nat}$ is only classically equivalent to transfinite (or strong) induction $\forall \PWF_{\Nat}$  on $\Nat$, according to which $\forall_{x \in \Nat}\big(\forall_{y < x}P(y) \To P(x)\big) \To \forall_{x \in \Nat}P(x)$, where $P(x)$ is a formula on $\Nat$ that respects $=_{\Nat}$.
Notice that the proof of the implication $\IND_{\Nat} \To \forall \PWF_{\Nat}$ is within $\INT$, since the proof of $P(0)$ requires $\EFQ$. 

%In order to work within $\MIN$, we accept both $\IND_{\Nat}$ \textit{and} $\exists \PWF_{\Nat}$ in our axiomatisation of the natural number. Hence, \textit{we add} $\exists \PWF_{\Nat}$ \textit{in the axioms of the natural number system in} $\BISH$.

\begin{proposition}\label{prp: equiv}
\normalfont (i)
\itshape $\forall \PWF_{\Nat}$ implies $\exists \PWF_{\Nat}$ constructively\footnote{We would like to thank Thierry Coquand for suggesting this proof to us.}.\\[1mm]
\normalfont (ii)
\itshape $\exists \PWF_{\Nat}$ implies classically $\forall \PWF_{\Nat}$ and $\LNP$.\\[1mm]
\normalfont (iii)
\itshape $\exists \PWF_{\Nat}$ implies constructively that for every sequence $\alpha \colon \Nat \to \Nat$, there is $i \in \Nat$ with $\alpha(i) \leq \alpha(i+1)$.\\[1mm]
\normalfont (iv)
\itshape The non-existence of an infinite descending sequence in $\Nat$ implies classically $\exists \PWF_{\Nat}$.
\end{proposition}

\begin{proof}
(i) We suppose that $\exists_{x \in \Nat}(Q(x)$, and let $n_0 \in \Nat$, such that $Q(n_0)$ holds. We also suppose that 
$$\forall_{x \in \Nat}\big(Q(x) \To (P(x) \vee \exists_{y < x}(Q(y)))\big).$$
 Let the predicate
$$R(n) :\TOT \bigg[Q(n) \ \& \ \forall_{x \in \Nat}\big(Q(x) \To (P(x) \vee \exists_{y < x}(Q(y)))\big)\bigg] \To \exists_{x \in \Nat}(P(x)).$$
It suffices to show that $\forall_{x \in \Nat}\big(\forall_{y < x}R(y) \To R(x)\big)$, since then by $\forall \PWF_{\Nat}$ we get $\forall_{x \in \Nat}R(x)$, and then the conlusion $R(n_0)$ tohether with the first two hypotheses imply with Modus Ponens the required formula $\exists_{x \in \Nat}(P(x))$.
Let $x \in \Nat$ and let 
$$\forall_{y < x}\bigg(\bigg[Q(y) \ \& \ \forall_{u \in \Nat}\big(Q(u) \To (P(u) \vee \exists_{w < u}(Q(w)))\big)\bigg] \To \exists_{z \in \Nat}(P(z))\bigg).$$
We show 
$$R(x) :\TOT \bigg[Q(x) \ \& \ \forall_{v \in \Nat}\big(Q(v) \To (P(v) \vee \exists_{k < v}(Q(k)))\big)\bigg] \To \exists_{z \in \Nat}(P(z)).$$
Let $Q(x) \ \& \ \forall_{v \in \Nat}\big(Q(v) \To (P(v) \vee \exists_{k < v}(Q(k)))\big)$. Hence, 
$Q(x)$ or $\exists_{k < x}(Q(k))$. In the first case, we get immediately what we want, while in the second case we use the inductive hypothesis for $Q(k)$.\\
%To prove the implication $\forall \PWF_{\Nat} \To \exists \PWF_{\Nat}$, let $\exists_{x \in \Nat}Q(x)$ and $\forall_{x \in \Nat}\big(Q(x) \To \big[P(x) \vee \exists_{y \in \Nat}\big(y < x \ \& \ Q(y)\big)\big]\big)$, and suppose that $\neg_{\Nat}\exists_{x \in \Nat}P(x)$. Let $x_0 \in \Nat$ with $Q(x_0)$. By hypothesis there is $x_1 < x_0$, such that $Q(x_1)$. Clearly, an infinite descending sequance in $\Nat$ is formed with dependent choice. $\forall \PWF_{\Nat}$ implies in $\MIN$ the standard induction principle $\IND_{\Nat}$, which in turns implies constructively the non existence of an infinite descending sequence $\alpha$ in $\Nat$, since it implies the existence of an index $i \in \Nat$ with $\alpha(i) \leq \alpha(i+1)$ (see~\cite{Pe22}, Proposition 2.1).\\
(ii) To prove the implication $\exists \PWF_{\Nat} \To \forall \PWF_{\Nat}$, we suppose that $\forall_{x \in \Nat}\big(\forall_{y < x}P(y) \To P(x)\big)$ and that $\neg_{\Nat} \forall_{x \in \Nat}P(x)$, hence classically $\exists_{x \in \Nat}\neg P(x)$. If $Q{'}(x) :\TOT \neg P(x)$ and $P{'}(x) :\TOT P(x) \ \& \ \neg P(x)$, then by hypothesis we have that $\exists_{x \in \Nat}Q{'}(x)$. Let $x \in \Nat$ with $P{'}(x)$. If $x =_{\Nat} 0$, then since $P(0)$ holds trivially, we get $P{'}(0)$. If $x > 0$, then by the hypothesis of $\forall \PWF_{\Nat}$ there is $y < x$ with $Q{'}(y)$. By the conclusion of $\exists \PWF_{\Nat}$ for $Q{'}(x)$ and $P{'}(x)$ we have that $\exists_{x \in \Nat}(P(x) \wedge \neg P(x))$, hence $\neg \neg_{\Nat} \forall_{x \in \Nat}P(x)$, and by double negation elimination we get $\forall_{x \in \Nat}P(x)$.
The fact that $\exists \PWF_{\Nat}$ implies classically $\LNP$ follows from the fact that $\forall \PWF_{\Nat}$ is classically equivalent to $\LNP$ (see~\cite{TvD88}, p.~129).\\
(iii) Let $\alpha \colon \Nat \to \Nat$. We define $Q_{\alpha}(x) :\TOT \exists_{n \in \Nat}(x =_{\Nat} \alpha(n))$ and $P_{\alpha}(x) :\TOT \exists_{n \in \Nat}\big(x =_{\Nat} \alpha(n) \ \& \ \alpha(n) \leq \alpha(n+1)\big)$. Trivially, $Q_{\alpha}(a_0)$ holds. Let $x, n \in \Nat$, such that $x =_{\Nat} \alpha(n)$. By (I1), if $\alpha(n) \leq \alpha(n+1)$, then $P_{\alpha}(x)$ holds. If $\alpha(n) > \alpha(n+1)$, then $Q_{\alpha}(\alpha(n+1))$. As the hypothesis of $\exists \PWF_{\Nat}$ holds for $Q_{\alpha}$ and $P_{\alpha}$, we get $x, n \in \Nat$, such that $x =_{\Nat} \alpha(n) \ \& \ \alpha(n) \leq \alpha(n+1)\big)$.\\
(iv) We work as in the classical proof of the implication $\forall \PWF_{\Nat}\To \exists \PWF_{\Nat}$ in case (i).
\end{proof}

%The classical well-foundedness of $\Nat$, according to which, there is no infinite descending sequence in $\Nat$, implies classically $\exists \PWF_{\Nat}$: suppose $\neg_{\Nat} \exists_{x \in \Nat}P(x)$, and let $x_0 \in \Nat$ with $Q(x_0)$. By hypothesis there is $x_1 < x_0$, such that $Q(x_1)$. Clearly, using dependent choice an infinite descending sequance in $\Nat$ is formed. 
%Next we show that $\exists \PWF_{\Nat}$ implies constructively that for every sequence $\alpha \colon \Nat \to \Nat$, there is an index $i \in \Nat$, such that $\alpha(i) \leq \alpha(i+1)$ (which in turn implies constructively that there is no infinite descending sequance in $\Nat$). 

The principle $\exists \PWF_{\Nat}$ can be used in order to avoid classical logic in the proof of fundamental arithmetical facts\footnote{I would like to thank P. Schuster for suggesting to me Proposition~\ref{prp: prime} as a case-study for $\exists \PWF_{\Nat}$.}. A standard classical proof of the divisibility of a natural number $n > 1$ by a prime number $p$ employs $\LNP$: if $\D P$ is the set of prime numbers, then one supposes that the set $A := \{x \in \Nat \mid  x > 1 \ \& \ \forall_{p \in \D P}(p \nmid x)\}$ is non-empty, and by $\LNP$ $A$ has a least element, through which a contradiction is induced. Let $\Prime(x) :\TOT x > 1 \ \& \ \forall_{y \in \Nat}(y \mid x \To y =_{\Nat} 1 \vee y =_{\Nat} x)$, and $\Coprime(x) :\TOT  x > 1 \ \& \ \exists_{y \in \Nat}(y \mid x \ \& \ y \neq_{\Nat} 1 \ \& \ y \neq_{\Nat}  x) \TOT x > 1 \ \& \ \exists_{y \in \Nat}(y \mid x \ \& \ 1 < y \ \& \ y <  x)$. Clearly\footnote{More generally, for every formula $A$ in primitive recursive arithmetic we have that $A \vee \neg_{\Nat} A$ is provable within it (see~\cite{TvD88}, p.~125).}, $\forall_{x > 1}\big(\Prime(x) \vee \Coprime(x)\big)$. If $x > 1$, such that $\Prime(x)$, then $x \in \D P$ with $x \mid x$. Thus, it suffices to prove the coprime case.

\begin{proposition}[$\MIN$]\label{prp: prime}
If $n \in \Nat$, such that $\Coprime(n)$, then there is $p \in \D P$ with $p \mid n$.
\end{proposition}

\begin{proof}
Let $Q_n(x) :\TOT 1 < x \ \& \ x \mid n$ and $P_n(x) :\TOT \Prime(x) \ \& \ x \mid n$. Clearly, $\Coprime(n) \To \exists_{x \in \Nat}Q_n(x)$. Let $x \in \Nat$, such that $Q_n(x)$. If $\Prime(x)$, then $P_n(x)$. If $\Coprime(x)$, then there is $y \in \Nat$ with $1 < y < x$ and $y \mid x$. Since by hypothesis $x \mid n$, we have that $y \mid n$, hence $y < x$ with $Q_n(y)$. By the conclusion of $\exists \PWF_{\Nat}$ there is $x \in \Nat$, such that $\Prime(x)$ and $x \mid n$.
\end{proof}

\section{Complemented subsets of $\Nat$}
\label{sec: cs}

Mathematics is more informative when weak negation is avoided in the definition of its concepts. If weak negation is involved in the definition of a mathematical concept, a strong version of this concept that avoids weak negation suits better to constructive study. 
For example, if $A \subseteq \Nat$, its \textit{weak} and \textit{strong complement} are the following extensional subsets of $\Nat$, respectively
$$A^{\neg_{\Nat}} :=   \big\{x \in \Nat \mid \forall_{a \in A}\big(\neg_{\Nat} (x =_{\Nat} a)\big)\big\},$$
$$A^{\neq_{\Nat}} := \{x \in \Nat \mid \forall_{a \in A}(x \neq_{\Nat} a)\}.$$
The \textit{weak empty} subset of $\Nat$ and the \textit{strong empty} subset of $\Nat$ are defined, respectively,  by
$$\emptyset_{\Nat} := \{x \in \Nat \mid \neg(x =_{\Nat} x)\},$$
$$\emptys_{\Nat} := \{x \in \Nat \mid x \neq_{\Nat} x\}.$$
We call	$A \in \C E(X)$ \textit{weakly empty}, if $A \subseteq \emptyset_{\Nat}$, and \textit{strongly empty}, if $A \subseteq \emptys_{\Nat}$.  Of course, due to the equivalence $\neg_{\Nat}(x =_{\Nat} y) \TOT x \neq_{\Nat} y$ the weak and the strong versions of these concepts for $\Nat$ coincide, although this is not the case for an arbitrary set with an inequality. Here, we keep the distinction, in order to be compatible with the more general theory of sets with an inequality in $\BISH$. 
The \textit{strong overlap} relation between subsets $A, B$ of $\Nat$ is defined by 
$$A \between B :\TOT \exists_{x \in A} \exists_{y \in B}\big(x =_{\Nat} y\big).$$
In section~\ref{sec: CLNP} we formulate 
%the constructive version 
$\CLNP$
% of $\LNP$ 
for \textit{complemented subsets} of $\Nat$ i.e., pairs $\B A := (A^1, A^0)$ of extensional subsets $A^1, A^0$ of $\Nat$ which are \textit{strongly disjoint}\footnote{\textit{Weakly complemeted subsets} of a set $X$ are defined as pairs $(A^1, A^0)$ of subsets of $X$ that are \textit{weakly disjoint} i.e., $\forall_{x \in A^1}\forall_{y \in A^0}\big(\neg(x =_{X} y)\big)$.}, in symbols $A^1 \Disj A^0$, where,  $$A^1 \Disj A^0 :\TOT \forall_{x \in A^1}\forall_{y \in A^0}\big(x \neq_{\Nat} y\big).$$
We call $\B A$ \textit{total}, if $\dom(\B A) := A^1 \cup A^0 = X$. If $n \in \Nat$, then by $(\Ineq_4)$ the \textit{complemented point} $\B n := \big(\{n\}, \{n\}^{\neq_{\Nat}}\big)$ is a total complemented subset of $\Nat$. 	We denote by $\C E^{\Disj}(\Nat)$ the totality of
 complemented subsets of $\Nat$.
 If $\B A, \B B \in \C E^{\Disj}(\Nat)$, let
 $$\B A \subseteq \B B : \TOT A^1 \subseteq B^1 \ \& \ B^0 \subseteq A^0, \ \ \ \ \ \B A =_{\mathsmaller{\C E^{\mathsmaller{\Disj}} (\Nat)}} \B B : \TOT \B A \subseteq \B B \ \& \ \B B \subseteq \B A.$$
If the elements of $A^1$ are the ``provers'' of $\B A$ and the elements of $A^0$ are the ``refuters'' of $\B A$, then the inclusion $\B A \subseteq \B B$ means\footnote{See also~\cite{Sh22} for a connection between Bishop's complemented subsets and the categorical Chu construction.} that all provers of $\B A$ prove $\B B$ and all refuters of $\B B$ refute $\B A$ i.e., $\B B$ has more provers and less refuters than $\B A$. The pair $\big(\{n\}, A^0\big)$, where $A^0$ is a proper subset of $\{n\}^{\neq_{\Nat}}$, is a simple example of a non-total complemented subset of $\Nat$. Next we show that there are complemented subsets of $\Nat$ that we cannot accept constructively to be total, although classically they are.

\begin{example}\label{ex: nontotal}
\normalfont
If $P$ is a formula as it is indicated in $\PEM_{\Nat}$, let the following subsets of $\Nat$:
\begin{align*}
P^1 & := \{x \in \Nat \mid x =_{\Nat} 1\} \cup \{x \in \Nat \mid x = 0 \ \& \ P\} \\
& =_{\mathsmaller{\C E(\Nat)}} \{x \in \Nat \mid x =_{\Nat} 1 \vee (x =_{\Nat} 0 \ \& \ P)\},
\end{align*}
\begin{align*}
P^0 & := \{x \in \Nat \mid x =_{\Nat} 1\}^{\neq_{\Nat}} \cap \{x \in \Nat \mid x \neq_{\Nat} 0 \vee \neg_{\Nat} P\}\\
&  =_{\mathsmaller{\C E(\Nat)}} \{x \in \Nat \mid x \neq _{\Nat} 1 \ \& \ (x \neq_{\Nat} 0 \vee \neg_{\Nat} P)\}\\
& =_{\mathsmaller{\C E(\Nat)}} \{x \in \Nat \mid (x \neq _{\Nat} 1 \ \& \ x \neq_{\Nat} 0) \vee (x \neq _{\Nat} 1 \ \& \ 
\neg_{\Nat} P)\}.
\end{align*}
First, we show that $P^1 \Disj P^0$. Let $x^1 \in P^1$ and $x^0 \in P^0$. If $x^1 =_{\Nat} 1$, then let first $x^0 \neq_{\Nat} 1 \ \& \ x^0 \neq_{\Nat} 0$. By the extensionality of $\neq_{\Nat}$ we get $x^1 \neq_{\Nat} x^0$. If $x^0 \neq _{\Nat} 1 \ \& \ \neg_{\Nat} P$, then we work similarly. If $x^1 =_{\Nat} 0 \ \& \ P$, then let first $x^0 \neq_{\Nat} 1 \ \& \ x^0 \neq_{\Nat} 0$. Again by the extensionality of $\neq_{\Nat}$ we get $x^1 \neq_{\Nat} x^0$.  If $x^0 \neq _{\Nat} 1 \ \& \ \neg_{\Nat} P$, then by $P$ and $\neg_{\Nat} P$ we get $\bot_{\Nat}$. By Proposition~\ref{prp: botnat}(iv) we get the required inequality $x^1 \neq_{\Nat} x^0$. Next, we show that if $0 \in P^1 \cup P^0$, then $P \vee \neg_{\Nat} P$ holds. If $0 \in P^1$, then $P$ holds. If $0 \in P^0$, then $0 \neq_{\Nat} 1 \ \& \ \neg_{\Nat} P$ holds, hence $\neg_{\Nat} P$ holds. 
%Again here we use that the disjunction holds, the left disjunct is locally negated, hence the right disjunct holds
\end{example}

\begin{definition}\label{def: lpart}
If $\B A := (A^1, A^0) \in \C E^{\Disj}(\Nat)$
% is a complemented subset of $\Nat$ 
 and $x \in \dom(\B A)$, then the \textit{downset } $\C D_{\B A}(x)$ of $x$ in $\B A$ is 
 %defined by
$$\C D_{\B A}(x) := \{y \in \dom(\B A) \mid y < x\}.$$
\end{definition}

We can show within $\MIN$ that $\C D_{\B A}(0)$ is strongly empty i.e., $\C D_{\B A}(0) \subseteq \emptys_{\Nat}$: If $y \in A^1 \cup A^0$ with $y < 0$, then, since $y \geq 0$, we get $y < y$, and hence by (I3) we get $y \neq_{\Nat} y$. The inclusion $\emptys_{\Nat} \subseteq \C D_{\B A}(0)$ can be shown within $\MIN$ if\footnote{In~\cite{MWP25} the complemented subsets of $\Nat$ that satisfy the stronger property ``the implication $\bot_{\Nat} \To A^1(x) \wedge A^0(x)$ is provable in $\MIN$'' are shown to form a swap algebra of type $(\tii)$, a generalisation of a Boolean algebra (see~\cite{MWP24}).} the implication $\bot_{\Nat} \To A^1(x) \vee A^0(x)$ can be shown within $\MIN$; if $x \in \Nat$ with $x \neq_{\Nat} x$, then by $(\Ineq_1)$ we get $\bot_{\Nat}$, and hence by hypothesis $x \in \dom(\B A)$. Moreover, if $x \neq_{\Nat} x$, then by (I4) we get $x < x$, and since $0 =_{\Nat} x$ (Proposition~\ref{prp: botnat}(ii)), by the extensionality of $<$ we get $x < 0$. Clearly, if $\B A$ is total, then the equality $\C D_{\B A}(0) =_{\mathsmaller{}\mathsmaller{}\C E(\Nat)} \emptys_{\Nat}$ is shown within $\MIN$.

%first we show that $\C D_{\B A}(0)$ is strongly empty i.e., $\C D_{\B A}(0) \subseteq \emptys_{\Nat}$. If $y \in \Nat$ with $y < 0$, then with $x \neq_{\Nat} x$. By $(\Ineq_1)$ we get $\bot_{\Nat}$. As $1 > 0$, by the extensionality of $>$ and the equality $0 =_{\Nat} 1$ we get $0 > 0$. As by Proposition~\ref{prp: botnat}(ii) $x =_{\Nat} 0$, by the extensionality of $>$ we also get $0 > x$.
%We structure this paper as follows:
%%\vspace{-3mm}
%\begin{itemize}
%	\item In section
%	
%	
%\end{itemize}
 The proof of Proposition~\ref{prp: invcs} is straightforward. In the general case of a function $f \colon X \to Y$ we need $f$ to be strongly extensional, in order to inverse the complemented subsets of $Y$. In the case of a function $f \colon \Nat \to \Nat$ though, by Proposition~\ref{prp: se}(i) strong extensionality of $f$ is provable.
% and it is omitted as it is straightfoward.

\begin{proposition}\label{prp: invcs}
Let $f \colon \Nat \to \Nat$ be a function. If $\B B := (B^1,B^0)$ is in $\C E^{\Disj}(\Nat)$, then $f^{-1}(\B B) := \big(f^{-1}(B^1), f^{-1}(B^0)\big)$ is in $\C E^{\Disj}(\Nat)$.
\end{proposition}

%\begin{proof}
%(i) It follows immediately from $(\Ineq_4)$.\\
%(ii) 
%\end{proof}

% In a proposition we write $(\INT)$ to denote that its proof is within intuitionistic logic. 
%Otherwise, all proofs presented here are within minimal logic $(\MIN)$.

\section{The constructive least number principle $\CLNP$}
\label{sec: CLNP}

\textit{Throughout this section} $\B A := (A^1, A^0)$ \textit{is a complemented subset of} $\Nat$.

\begin{definition}\label{def: mcs}
We call $\B A$ downset located, if
$$\forall_{x^1 \in A^1}\big(\C D_{\B A}(x^1) \subseteq A^0 \vee \C D_{\B A}(x^1) \between A^1\big).$$
\end{definition}

\begin{example}\label{ex: lpl}
\normalfont
(i) If $\B A$ is total, then $\B A$ is downset located. If $0 \in A^1$, then $\C D_{\B A}(0) =_{\C E(\Nat)} \emptys_{\Nat}$, and with $\EFQ$ we get\footnote{For many concrete subsets $A^0$ of $\Nat$ the inclusion $\emptys_{\Nat} \subseteq A^0$ can be shown within $\MIN$. See the proof in Example~\ref{ex: lpl}(ii).} $\emptys_{\Nat} \subseteq A^0$. If $x^1$ is a non-zero element of $A^1$, then $\C D_{\B A}(0) =_{\C E(\Nat)} \{0, \ldots, x^1-1\}$ and the required disjunction holds because for every $i \in \{0, \ldots, x^1-1\}$ we have that $i \in A^1 \vee i \in A^0$. Using this argument and classical logic, then all complemented subsets of $\Nat$ are downset located.\\
(ii) Let $A^1 := \{x \in \Nat \mid x =_{\Nat} 2\} =: \{2\}$ and $A^0 := \{3\}$. Then $\B A$ is downset located, but not total. Working as in example (i), we have that $\C D_{\B A}(2) =_{\C E(\Nat)} \emptys_{\Nat}$. The inclusion $\emptys_{\Nat} \subseteq A^0$ is shown within $\MIN$ as follows: if $x \in \emptys_{\Nat}$, then by $(\Ineq_1)$ we get $\bot_{\Nat}$, and by Proposition~\ref{prp: botnat}(ii)
 %we conclude that 
 $3 =_{\Nat} 0 =_{\Nat} x$.\\
 (iii) The complemented subset $\B P := (P^1, P^0)$, where $P^1, P^0$ are defined in Example~\ref{ex: nontotal}, cannot be accepted constructively to be downset located. As $1 \in P^1$, its downset  $\C D_{\B P}(1) := \{x \in P^1 \cup P^0 \mid x < 0\}$ overlaps with $P^1$ only if $0 \in P^1$ and $P$ holds, and it is included in $P^0$ if there is $x \in P^1 \cup P^0$, such that $x < 1$ and $x \in P^0$, hence $\neg_{\Nat} P$.
\end{example}

Next, we show that there is a plethora of downset located subsets of $\Nat$, induced by appropriate monotone functions from $\Nat$ to $\Nat$. 
%Notice that the hypothesis of surjectivity on $B^1$ mentioned in the next result can be replaced by the hypothesis of $f$ being a \textit{simulation}\footnote{See \url{https://ncatlab.org/nlab/show/simulation}.} from $(\Nat, <)$ to $(\Nat, <)$ i.e., $f$ is a monotone function, such that $\forall_{x, y \in \Nat}\big(y < f(x) \To \exists_{z < x}\big(y = f(z)\big)\big)$. 

\begin{proposition}\label{prp: invlpc}
Let $f \colon \Nat \to \Nat$ be a monotone function and $\B B$ a downset located complemented subset of $\Nat$. If $f$ is onto $B^1$, then  $f^{-1}(\B B)$ is also downset located.
\end{proposition}

\begin{proof}
By the hypothesis on $\B B$ we have that $\forall_{y^1 \in B^1}\big(\C D_{\B B}(y^1) \subseteq B^0 \vee \C D_{\B B}(y^1) \between B^1\big)$. Let $x^1 \in f^{-1}(B^1) \TOT f(x^1) \in B^1$. Hence, by the hypothesis on $\B B$ we get
$$\C D_{\B B}(f(x^1)) \subseteq B^0 \ \vee \ \C D_{\B B}f((x^1)) \between B^1.$$
We show that 
$$\C D_{f^{-1}(\B B}(x^1) \subseteq f^{-1}(B^0) \ \vee \ \C D_{f^{-1}(\B B}(x^1) \between f^{-1}(B^1).$$
First, we suppose that $\C D_{\B B}(f(x^1)) \subseteq B^0$ i.e., $\{u \in \dom(\B B) \mid u < f(x^1)\} \subseteq B^0$, and we show that $\C D_{f^{-1}(\B B}(x^1) \subseteq f^{-1}(B^0)$.
For that, let $w \in f^{-1}(B^1) \cup f^{-1}(B^0) \TOT f(w) \in \dom(\B B)$, such that $w < x^1$. By monotonicity of $f$ we get $f(w) < f(x^1)$, and hence $f(w) \in B^0$ i.e., $w \in f^{-1}(B^0)$. Hence, we showed that $\C D_{f^{-1}(\B B}(x^1) \subseteq f^{-1}(B^0)$. Next, we suppose that $\C D_{\B B}f((x^1)) \between B^1$ i.e., there is $u^1 \in B^1$ with $u^1 < f(x^1)$, and we show that $\C D_{f^{-1}(\B B}(x^1) \between f^{-1}(B^1)$ i.e., we find $w \in  f^{-1}(B^1) \cup f^{-1}(B^0)$ with $w < x^1$ and $f(w) \in B^1$. Since $f$ is onto $B^1$, there is $w \in \Nat$, such that $f(w) =_{\Nat} u^1$ i.e., $w \in f^{-1}(B^1)$. By the extensionalty of $<$ we get $f(w) < f(x^1)$, while by the extensionality of $B^1$ we get $f(w) \in B^1$. Since by 
Proposition~\ref{prp: se}(ii) $f$ is strongly monotone, we get $w < x^1$, and $\C D_{f^{-1}(\B B}(x^1) \between f^{-1}(B^1)$ is shown.
\end{proof}

It is also straightforward to show that if $f \colon \Nat \to \Nat$ is monotone, then
$$f\big(\C D_{f^{-1}(\B B}(x)\big) \subseteq \C D_{\B B}(f(x)) \ \ \ \& \ \ \  f^{-1}\big(\C D_{\B B}(f(x))\big) \subseteq \C D_{f^{-1}(\B B}(x).$$

\begin{definition}\label{def: min}
A natural number $\mu$ is a least element\footnote{If $\B A$ is total, then this definition is the complemented subset version of the standard definition of a least element $\mu$ of a subset $A$ of $\Nat$: $\mu \in A \ \& \ \forall_{x \in \Nat}\big(x < \mu \To \neg_{\Nat}(x \in A)\big)$ (see also~\cite{TvD88}, p.~129).} of $\B A$ if and only if 
$$\mu \in A^1 \ \& \ \forall_{x \in \dom(\B A)}\big(x < \mu \To x \in A^0\big).$$
\end{definition}

\begin{corollary}[$\MIN$]\label{cor: min1}
Let $\mu, \nu \in \Nat$, such that $\mu$ and $\nu$ are least elements of $\B A$.\\[1mm]
\normalfont 
(i) \itshape $\forall_{x \in A^1}\big(x \geq \mu\big)$.\\[1mm]
\normalfont (ii) \itshape $\mu =_{\Nat} \nu$.\\[1mm]
\normalfont (iii) \itshape If $0 \in A^1$, then $0$ is the least element of $\B A$.
\end{corollary}

\begin{proof}
(i) If $x \in A^1$, then by dichotomy
 %we have that
  $x < \mu \vee x \geq \mu$. If $x < \mu$, then by the definition of $\mu$ we get $x \in A^0$, hence $x \neq_{\Nat} x$, and consequently $\bot_{\Nat}$. Hence, we get $\neg_{\Nat}(x < \mu)$. Consequently, we get $x \geq \mu$.\\
(ii) As $\mu, \nu \in A^1$, by case (i) we have that $\nu \geq \mu$ and $\mu \geq \nu$, hence $\mu =_{\Nat} \nu$.\\[1mm]
(iii) Let $x \in A^1 \cup A^0$ i.e., $x \in A^1 \vee x \in A^0$. If $x < 0$, then the hypothesis $x \in A^1$ implies by (i) that $x \geq 0$, hence $x < x$ and $x \neq_{\Nat} x$. Consequently, we get $\bot_{\Nat}$ i.e., $\neg_{\Nat}(x \in A^1)$. Thus, $x \in A^0$. 
\end{proof}

The following equivalence is our constructive least number principle $\CLNP$. If $0 \in A^1$, then by Corollary~\ref{cor: min1}(iii) $0$ is the least element of $\B A$. 

\begin{theorem}[$\MIN$]\label{thm: CLNP}
%	Let $\B A$ be inhabited i.e., there is
	Let $a^1 \in A^1$ with $a^1 > 0$. The following are equivalent:\\[1mm]
	\normalfont (i) \itshape $\B A$ has a least element.\\[1mm]
	\normalfont (ii) \itshape $\B A$ is downset located.
\end{theorem}

\begin{proof}
(i) $\To$ (ii): Let $\mu \in A^1$, such that $\forall_{x \in \dom(\B A)}\big(x < \mu \To x \in A^0\big)$.
Let $x^1 \in A^1$. By Corollary~\ref{cor: min1}(i) we get $x^1 \geq \mu$. By (I2) we have that $x^1 =_{\Nat} \mu$ or $x^1 > \mu$. If $x^1 =_{\Nat} \mu$, then by Definition~\ref{def: min} and the extensionality of $<$ we have that $\C D_{\B A}(x^1) =_{\mathsmaller{\mathsmaller{\C E(\Nat})}} \C D_{\B A}(\mu) \subseteq A^0$. If $x^1 > \mu$, then $\mu \in \C D_{\B A}(x^1) \cap A^1$ i.e., $\C D_{\B A}(x^1) \between A^1$.\\
(ii) $\To$ (i) (informally): By hypothesis we have that $\C D_{\B A}(a^1) \subseteq A^0 \vee \C D_{\B A}(a^1) \between A^1$. If $\C D_{\B A}(a^1) \subseteq A^0$, then $a^1$ is the least element of $\B A$. If $\C D_{\B A}(a^1) \between A^1$, let $a^2 \in A^1$ with $a^2 \in \C D_{\B A}(a^1) \cap A^1$ i.e., $a^2 \in A^1$ and $a^2 < a^1$. Again $\C D_{\B A}(a^2) \subseteq A^0 \vee \C D_{\B A}(a^2) \between A^1$, and we repeat the previous argument. After at most $(a^1 + 1)$-number of steps, we will have found the least element of $\B A$. \\
(ii) $\To$ (i) (formally): Let $Q_{\B A}(x) :\TOT x \in A^1$ and $P_{\B A}(x) :\TOT x \in A^1 \ \& \ \C D_{\B A}(x) \subseteq A^0$. By hypothesis we have that $Q_{\B A}(a^1)$. Let $x \in \Nat$ with $x \in A^1$. Since $\B A$ is downset located, we get $\C D_{\B A}(x) \subseteq A^0 \vee \C D_{\B A}(x) \between A^1$. In the first case, we get immediately $P_{\B A}(x)$. In the second case, we get some $y \in \Nat$ with $y < x$ and $y \in A^1 \TOT: Q_{\B A}(y)$. Hence, by $\exists \PWF_{\Nat}$ ther is $x \in \Nat$, such that $x \in A^1$ and $\C D_{\B A}(x) \subseteq A^0$. Clearly, $x$ is then the least element of $\B A$.
\end{proof}

The above proof of the implication (ii) $\To$ (i) can be seen as the constructive content of the corresponding classical proof of the least number principle for an arbitrary non-empty subset of $\Nat$, which employs the principle of the exluded middle. We can also use $\CLNP$ directly, in order to prove Proposition~\ref{prp: prime}. If $n \in \Nat$ with $n > 1$, let $\B P(n) := \big(P^1(n), P^0(n)\big)$, where
$$P^1(n) := \{x > 1 \mid x \mid n \ \& \ \Coprime(x)\}, \ \ \ P^0(n) := \{x > 1 \mid x \mid n \ \& \ \Prime(x)\}.$$
Clearly, $P^1(n) \Disj P^0(n)$. If $\Coprime(n)$, then $n > 0$, and trivially $n \in P^1(n)$. We show that $\B P(n)$ is downset located. If $x^1 \in P^1(n)$, then the disjunction $\C D_{\B P(n)}(x^1) \subseteq P^0 \vee \C D_{\B P(n)}(x^1) \between P^1$ follows by the decidability $(D) \ \forall_{x > 1}(\Prime(x) \vee \Coprime(x))$. By Theorem~\ref{thm: CLNP} $\B P(n)$ has a least element $\mu$ i.e., $\mu \in P^1(n)$ and $\forall_{x \in P^1(n) \cup P^0(n)}\big(x < \mu \To x \in P^0(n)\big)$. Since $\Coprime(\mu)$, there is 
$y \in \Nat$ with $1 < y < \mu$ and $y \mid \mu$. As $\mu \mid n$, we also have that $y \mid n$. By $(D)$ we have that $y \in P^1(n) \cup P^0(n)$. As $y < \mu$, we get $y \in P^0(n)$ i.e., $\Prime(y)$ and $y \mid n$.

\section{$\exists$-well-founded sets}
\label{sec: wfs}

In this section we generalise the $\exists$-well-foundedness of $\Nat$. As constructively every $\forall$-well-founded set in $\Nat$ is also $\exists$-well-founded, but not necessarily the converse,
it is meaningful to elaborate this coinductive notion of well foundedness independently from the standard inductive one.

%WE NEED TO FIND CAN WE ???? a set which is exists well founded but not forall? May be we cannot exhibit such a set. ?????????? 

\begin{definition}\label{def: rel}
A set with an extensional inequality and relation is a structure $\C X := (X, =_X, \neq_X, <_X)$, where $(X, =_X, \neq_X) \in \SetExtIneq$ and $x <_X x{'}$ is an extensional binary relation on $X$. We call $\C X$ dichotomous if $\forall_{x, x{'} \in X}\big(x \neq_X x{'} \To (x <_X x{'} \vee x{'} <_X x)\big)$, and we call $\C X$ strong if $\forall_{x, x{'} \in X}\big(x <_X x{'} \To x \neq_X x{'}\big)$. Let $(\SetExtIneqRel, \StrExtFunRel)$ be the category of sets with an extensional inequality and relation and of strongly extensional functions that preserve the corresponding relations i.e., if $\C Y := (Y, =_Y, \neq_Y, <_Y)$ is in $\SetExtIneqRel$ and $f \colon X \to Y$ is in $\StrExtFun$, we also have that $\forall_{x, x{'} \in X}\big(x <_X x{'} \To f(x) <_Y f(x{'})\big)$.
\end{definition}

By definition, $(\SetExtIneqRel, \StrExtFunRel) \leq (\SetExtIneq, \StrExtFun)$. By properties (I3, I4) in section~\ref{sec: nat} $\C N := (\Nat, =_{\Nat}, \neq_{\Nat}, <_{\Nat})$ is in $\SetExtIneqRel$ that is also strong and dichotomous. 

%Some remarks on the relation ....

\begin{definition}[$\exists$-well-founded sets]\label{def: ewfs}
Let $\C X := (X, =_X, \neq_X, <_X)$ be in $\SetExtIneqRel$. We say that $\C X$ is an $\exists$-well-founded set $(\exists$-$\wfs)$ if it satisfies the scheme $\exists \PWF_X:$ for every extensional\footnote{The hypothesis of extensionality on $Q(x)$ and $P(x)$ is crucial in the proof of Proposition~\ref{prp: sigmaset}.} formulas $Q(x), P(x)$ on $X$ with 
$$Q(x) \stackrel{\mathsmaller{<_X}} \To P(x) :\TOT \forall_{x \in X}\big(Q(x) \To [P(x) \vee \exists_{x{'} \in X}(x{'} <_X x \ \& \ Q(x{'})]\big),$$
then 
$$\exists_{x \in X}Q(x) \To \exists_{x \in X}P(x).$$
Let $(\exists \WFSet, \StrExtFunRel)$ be the category of $\exists$-$\wfs$ and of strongly extensional functions that preserve the given extensional relations. We call $\C X$ a $\forall$-well-founded set $(\forall$-$\wfs)$ if it satisfies the scheme of $\forall$-well-founded induction $\forall \PWF_X$. 
% instead of $\exists \PWF_X$.
\end{definition}

By $\exists \PWF_{\Nat}$  the above structure of naturals $\C N$ is an $\exists$-$\wfs$. Recall that $\C N$ is a $\forall$-$\wfs$ within $\INT$. Next, we prove some fundamental results on $\exists$-well-founded sets that also hold constructively for $\forall$-well-founded sets (see~\cite{MRR88, RS93}). Our proofs are interesting because they show that our notion of well-foundedness is sufficient and they provide a new algorithmic content of these results.
A subset $A$ of $X$, where $\C X \in \SetExtIneqRel$ \textit{has no minimal elements} if $\forall_{x \in A}\exists{y \in A}(y <_X x)$. As expected, we define the non-existence of minimal elements in a positive way, in order to avoid weak negation. All basic set-theoretic definitions on $\Nat$ that are included in section~\ref{sec: nat} are extended to arbitrary sets. For example, $A$ is strongly empty if $A \subseteq \emptys_X := \{x \in X \mid x \neq_X x\}$.

\begin{proposition}\label{prp: basic1}
Let $\C X := (X, =_X, <_X)$ be an $\exists$-$\wfs$, and let $A \subseteq X$.\\[1mm]
\normalfont (i)
\itshape $\forall_{x \in X}\big(\neg_{\Nat}(x <_X x)\big)$.\\[1mm]
\normalfont (ii)
\itshape If $A$ has no minimal elements, then $A$ is strongly empty.\\[1mm]
\normalfont (iii)
\itshape If $(x_n)_{n \in \Nat}$ is an infinite descending sequence in $X$, then $A := \{x_n \mid n \in \Nat\}$ is strongly empty.
\end{proposition}

\begin{proof}
(i) Let the extensional relations $Q(x) :\TOT x <_X x$ and $P(x) :\TOT \bot_{\Nat}$ on $X$. We show that
$Q(x) \stackrel{\mathsmaller{<_X}} \To P(x)$. If $x \in X$, such that $x <_X x$, then $\exists_{x{'} < x}(Q(x{'}))$, since we can take $x$ again. By $\exists \PWF_X$, if we suppose that there is $x_0 \in X$ with $x_0 <_X x_0$, then we get $\exists_{x \in X}\bot_{\Nat}$, and hence we get $\bot_{\Nat}$.\\
(ii) If $x_0 \in A$, let the extensional relations $Q_A(x)$, where $A := \{x \in X \mid Q_A(x)\}$  and $P(x) :\TOT x =_X x_0 \ \& \ x \neq_X x$ on $X$. We show that
$Q_A(x) \stackrel{\mathsmaller{<_X}} \To P_A(x)$. If $x \in X$, such that $x \in A$, since $A$ has no minimal elements there is $y < A$ with $y <_X x$ i.e., $\exists_{y <_X x}Q_A(y)$. By $\exists \PWF_X$ we get $\exists_{x \in X}(x =_X x_0 \ \& \ x \neq_X x)$. Hence by the extensionality of $\neq_X$ we conclude that $x_0 \neq_X x_0$ i.e., $x_0 \in \emptys_X$.\\
(iii) It follows immediately from case s (ii).
\end{proof}

If $\C X := (X, =_X, \neq_X, <_X) \in \SetExtIneqRel$ and $P(x)$ is an extensional formula on $X$, let the standard $\forall$-formulation of the well-foundedness of $\C X$:
$$(\forall \PWF_X) \ \ \ \ \ \ \ \ \ \ \ \ \ \ \ \ \ \ \ \ \ \ \ 
\forall_{x \in X}\big(\forall_{x{'} <_X x}P(x{'}) \To P(x)\big) \To \forall_{x \in X}P(x), \ \ \ \ \ \ \ \ \ \ \ \ \ \ \ \ \ \ \ \ \ \ \ \ \ \ \ \ \ \ \  \ \ \ \ \ \ \ \ \ \ \ \ \ \ \ \ \ \ \ $$
In~\cite{MRR88}, pp.~28-29, the following $(\forall, \vee)$-well-foundedness is given:
$$(\forall,\vee)-(\PWF_X) \ \ \ \ \ \ \ \ \ \ \ \ \ \ \ \ \ \ \ \ \ \ \ \forall_{x \in X}\big(P(x) \vee \exists_{x{'} <_X x}(P(x{'}) \To P(x)\big) \To \forall_{x \in X}P(x). \ \ \ \ \ \ \ \ \ \ \  \ \ \ \ \ \ \ \ \ \ \ \ \ \ \ \ \ \ $$

\begin{proposition}\label{prp: basic2}
	Let $\C X := (X, =_X, <_X)$, $\C Y := (Y, =_Y, <_Y)$ be in $\exists$-$\WFSet$, and let $A \subseteq X$.\\[1mm]
	\normalfont (i)
	If $\C Z := (Z, =_Z, \neq_Z,<_Z) \in\SetExtIneqRel$ and $f \colon Z \to X \in \StrExtFunRel$, then $\C Z$ is an $\exists$-$\wfs$. If $\C X$ is strong, then $\C Z$ is strong.\\[1mm]
	\normalfont (ii)
	\itshape If $=_A$ and $<_A$ are the restrictions of $=_X$ and $<_X$, respectively, then $\C A := (A =_A, <_A)$ is an $\exists$-$\wfs$.\\[1mm]
	\normalfont (iii)
	\itshape The product $\C X \times \C Y := (X \times Y, =_{X \times Y},\neq_{X \times Y}, <_{X \times Y})$, where 
	$$(x, y) \neq_{X \times Y} (x{'}, y{'}) :\TOT x \neq_X x{'} \vee y <_Y y{'},$$
	$$(x, y) <_{X \times Y} (x{'}, y{'}) :\TOT x <_X x{'} \ \& \ y <_Y y{'},$$
	is an $\exists$-$\wfs$. If $\C X$, or $\C Y$, is strong, then  $\C X \times \C Y$ is strong.\\[1mm]
	\normalfont (iv) 
	\itshape The sum $\C X + \C Y := (X + Y, =_{X + Y}, \neq_{X + Y}, <_{X + Y})$, where
	$$w \in X + Y :\TOT \exists_{x \in X}\big(w := (0, x)\big) \vee \exists_{y \in Y}\big(w := (1, y)\big),$$
	$$(i, z) =_{X + Y} (j, u) :\TOT (i =_{\D 2} j =_{\D 2} 0 \ \wedge \ z =_{X} u) \vee (i =_{\D 2} j =_{\D 2} 1 \ \wedge \ z =_{Y} u),$$
	$$(i, z) \neq_{X + Y} (j, u) :\TOT i \neq_{\D 2} j \vee (i =_{\D 2} j =_{\D 2} 0 \ \& \ z \neq_X u) \vee (i =_{\D 2} j =_{\D 2} 1 \ \& \ z \neq_Y u),$$
	$$(i, z) <_{X + Y} (j, u) :\TOT i <_{\D 2} j \vee (i =_{\D 2} j =_{\D 2} 0 \ \& \ z <_X u) \vee (i =_{\D 2} j =_{\D 2} 1 \ \& \ z <_Y u),$$
	is an $\exists$-$\wfs$.  If $\C X$ and $\C Y$ are strong $($dichotomous$)$, then  $\C X + \C Y$ is strong $($dichotomous$)$.
	
\end{proposition}

\begin{proof}
(i) Let $Q_Z(z)$ and $P_Z(z)$ be extensional formulas on $Z$, such that 
$$Q_Z(z) \stackrel{\mathsmaller{\mathsmaller{\mathsmaller{<_{Z}}}}} \Rightarrow P_Z(z) :\TOT \forall_{z \in Z}\big(Q_Z(z) \To \big[P_Z(z) \vee  \exists_{z{'} <_Z z}Q_Z(z{'})\big]\big).$$
We define the following extensional relations on $X$:
$$Q_{f, X}(x) :\TOT \exists_{z \in Z}\big(f(z) =_X x \ \& \ Q_Z(z)\big),$$
$$P_{f, X}(x) :\TOT \exists_{z \in Z}\big(f(z) =_X x \ \& \ P_Z(z)\big),$$
and we show that 
$$Q_{f,X}(x) \stackrel{\mathsmaller{\mathsmaller{\mathsmaller{<_{X}}}}} \Rightarrow P_{f,X}(x) :\TOT \forall_{x \in X}\big(Q_{f,X}(x) \To \big[P_{f,X}(x) \vee  \exists_{x{'} <_X x}Q_{f,X}(x{'})\big]\big).$$
Let $x \in X$ and $z \in Z$, such that $f(z) =_X x$ and $Q_Z(z)$. If $P_Z(z)$, then we get $P_{f,X}(x)$. If $z{'} <_Z z$ with $Q_Z(z{'})$, let $x{'} := f(z{'}) \in X$. Since $f$ respects the relations, we get $x{'} := f(z{'}) <_X f(z) =_X x$ and $Q_{f,X}(x{'})$. Next we suppose that $\exists_{z \in Z}Q_Z(z)$, hence $\exists_{x \in X}Q_{f,X}(x)$. By $\exists$-$\PWF_X$ we get $\exists_{x \in X}P_{f,X}(x)$, hence $\exists_{z \in Z}P_Z(z)$. Moreover, if $z <_Z z{'}$, then $f(z) <_X f(z{'})$, and since $\C X$ is strong, we get $f(z) \neq_X f(z{'})$. Since $f$ is strongly extensional, we get $z \neq_Z z{'}$, and hence $\C Z$ is strong.\\
(ii) It follows from case (i), since the embedding of $A$ into $X$ is in $\StrExtFunRel$.\\
(iii)  It follows from case (i), since the projection function\footnote{Notice that in the case of the product it suffices one of the two sets to be an $\exists$-$\wfs$.} $\pr_X \colon X \times Y \to X$ is in $\StrExtFunRel$.\\
(iv) We only prove that $\exists \PWF_{X + Y}$. Let $Q(w)$ and $P(w)$ be extensional formulas on $X + Y$, such that 
$Q(w) \stackrel{\mathsmaller{\mathsmaller{\mathsmaller{<_{X+Y}}}}} \Longrightarrow P(w)$. Suppose first that there is $x_0 \in X$ with $Q((0, x_0))$. Let $P_X(x) :\TOT P((0,x))$ and $Q_X(x) :\TOT Q((0,x))$ formulas on $X$. Clearly, the extensionality of $P$ and $Q$ implies the extensionality of $P_X$ and $Q_X$, respectively. We show that $Q_X(x) \stackrel{\mathsmaller{\mathsmaller{\mathsmaller{<_{X}}}}} \To P_X(x)$. Let $x \in X$ with $Q_X(x)$. By the hypothesis $Q(w) \stackrel{\mathsmaller{\mathsmaller{\mathsmaller{<_{X+Y}}}}} \Longrightarrow P(w)$ we get
$$P((0, x)) \TOT: P_X(x) \ \vee \ \exists_{(j, u) \in X+Y}\big((j, u) <_{\mathsmaller{X+Y}} (0, x) \ \& \ Q((j, u))\big).$$
If the right disjunct holds, then $j =_{\D 2} 0$, as $j <_{\D 2} 0 \To \bot_{\Nat}$, and $u <_X x$ with $Q((0, u))$. Hence, $Q_X(x) \stackrel{\mathsmaller{\mathsmaller{\mathsmaller{<_{X}}}}} \To P_X(x)$ is shown. As $Q((0, x_0)) \To \exists_{x \in X}Q_X(x)$, by $\exists \PWF_X$ we get $\exists_{x \in X}P_X(x)$, thus $\exists_{w \in X+Y}Q(w)$. Next we suppose that there is $y_0 \in Y$ with $Q(1, y_0)$. Let the following extensional formulas on $Y$:
$$P_Y(y) :\TOT P((1,y)) \vee \exists_{x \in X}Q((0, x)), \ \ \ \ Q_y(y) :\TOT Q((1,y)).$$ 
%	formulas on $Y$. 
	%Clearly, the extensionality of $P$ and $Q$ implies the extensionality of $P_Y$ and $Q_Y$, respectively. 
We show that $Q_Y(y) \stackrel{\mathsmaller{\mathsmaller{\mathsmaller{<_{Y}}}}} \To P_Y(y)$. 
Let $y \in Y$ with $Q_Y(y)$. By the hypothesis $Q(w) \stackrel{\mathsmaller{\mathsmaller{\mathsmaller{<_{X+Y}}}}} \Longrightarrow P(w)$ we get
$$P((1, y)) \ \vee \ \exists_{(j, u) \in X+Y}\big((j, u) <_{\mathsmaller{X+Y}} (1, y) \ \& \ Q((j, u))\big).$$
Trivially, $P((1, y)) \To P_Y(y)$. If the right disjunct holds, then let first the case $j =_{\D 2} 0$ and $u \in X$, such that $Q((0, u))$. Hence, $\exists_{x \in X}Q((0, x))$, and trivially $\exists_{x \in X}Q((0, x)) \To P_Y(y)$. The other case is that $j =_{\D 2} 1$ and $u \in Y$, such that $u <_Y y$ and $Q((1, u)) \TOT: Q_Y(u)$. Hence, $Q_Y(y) \stackrel{\mathsmaller{\mathsmaller{\mathsmaller{<_{Y}}}}} \To P_Y(y)$ is shown. As $Q((1, y_0)) \To \exists_{y \in Y}Q_Y(y)$, by $\exists \PWF_Y$ we get $\exists_{y \in Y}P_Y(y)$. Thus either there is $y \in Y$ with $P((1,y))$, hence $\exists_{w \in X+Y}P(w)$, or $\exists_{x \in X}Q((0, x))$. In the latter case we work as in the first part of the proof. The last part of case (iv) is straightforward to show.
% :\TOT \forall_{x \in X}\big(Q(x) \To [P(x) \vee \exists_{x{'} \in X}(x{'} <_X x \ \& \ Q(x{'})]\big),$$
\end{proof}

By Proposition~\ref{prp: basic2}(ii) the structure of booleans $\C B := (\D 2, =_{\D 2}, \neq_{\D 2}, <_{\D 2})$ is in $\exists \WFSet$. Clearly, the projections $\pr_X \colon X \times Y \to X$ and $\pr_Y \colon X \times Y \to Y$ are in $\StrExtFunRel$ and $\C X \times \C Y$ is a product of 
$\C X$ and $\C Y$ in $(\exists \WFSet, \StrExtFunRel)$. Similarly, the injections $\inj_X \colon X \to X + Y$ and $\inj_Y \colon Y \to X + Y$ are in $\StrExtFunRel$ and $\C X + \C Y$ is a coproduct of 
$\C X$ and $\C Y$ in $(\exists \WFSet, \StrExtFunRel)$. Next, we show that the product of two $\exists$-$\wfs$ with the lexicographic order is an $\exists$-$\wfs$.

\begin{proposition}\label{prp: lex}
Let $\C X := (X, =_X, <_X)$, $\C Y := (Y, =_Y, <_Y)$ be in $\exists$-$\WFSet$, and let
$$(x, y) <_{\mathsmaller{\lex}} (x{'}, y{'}) :\TOT x <_X x{'} \vee (x =_X x{'} \ \& \ y <_Y y{'}).$$
\normalfont (i)
\itshape $\C X \times_{\lex} \C Y := (X \times Y, =_{X \times Y}, \neq_{X \times Y}, <_{\mathsmaller{\lex}})$ is an $\exists$-$\wfs$.\\[1mm]
 \normalfont (ii)
 \itshape If $\C X$ and $\C Y$ are strong, then $\C X \times_{\mathsmaller{\lex}} \C Y$ is strong.\\[1mm]
  \normalfont (iii)
 \itshape If $\C X \times_{\mathsmaller{\lex}} \C Y$ is strong and $X, Y$ are inhabited, then $\C X$ and $\C Y$ are strong.\\[1mm]
   \normalfont (iv)
 \itshape If $\C X \times_{\mathsmaller{\lex}} \C Y$ is dichotomous and $X, Y$ are inhabited, then $\C X$ and $\C Y$ are dichotomous.
\end{proposition}

\begin{proof}
(i) Clearly, the extensionality of $<_X$ and $<_Y$ imply the extensionality of $<_{\mathsmaller{\lex}}$. Let $Q((x, y))$ and $P((x,y))$ extensional formulas on $X \times Y$, such that
$$Q((x,y)) \stackrel{\mathsmaller{\mathsmaller{\mathsmaller{<_{\lex}}}}} \Rightarrow P(x,y) :\TOT \forall_{(x,y) \in X \times Y}\big(Q((x,y)) \To \big[P((x,y)) \vee  \exists_{(x{'}, y{'}) <_{\mathsmaller{\lex}} (x,y)}Q((x{'}, y{'}))\big]\big).$$
We suppose that $\exists_{(x, y) \in X \times Y}Q((x,y))$, and we show that $\exists_{(x, y) \in X \times Y}P((x,y))$. Let $(x_0, y_0) \in X \times Y$, such that $Q((x_0,y_0))$. Let the following extensional formulas on $Y$:
$$Q_Y(y) :\TOT Q((x_0, y)),$$
$$P_Y(y) :\TOT P((x_0, y)) \vee \exists_{x <_X x_0}\exists_{z \in Y}Q((x, z)).$$
Since $Q((x_0, y_0))$, we have that $\exists_{y \in Y}Q_Y(y)$ holds. Next we show that 
$$Q_Y(y) \stackrel{\mathsmaller{\mathsmaller{\mathsmaller{<_{Y}}}}} \Rightarrow P_Y(y) :\TOT \forall_{y \in Y}\big(Q_Y(y) \To \big[P_Y(y) \vee \exists_{y{'} <_Y y}Q_Y(y{'})\big]\big).$$
Let $y \in Y$, such that $Q(x_0, y)$. By the hypothesis $Q((x,y)) \stackrel{\mathsmaller{\mathsmaller{\mathsmaller{<_{\lex}}}}} \Rightarrow P(x,y)$ we get
$$P((x_0,y))  \vee \exists_{(x{'}, y{'}) \in X \times Y}(x{'} <_X x_0 \ \& \ Q(x{'}, y{'})) \vee \exists_{(x{'}, y{'}) \in X \times Y}(x{'} =_X x_0 \ \& \ y{'} <_Y y \ \& \ Q(x{'}, y{'})).$$
The first two disjuncts trivially imply $P_Y(y)$, while the last disjunct, together with the extensionality of $Q$ imply $ \exists_{y{'} <_Y y}Q((x_0, y{'}))$ i.e.,  $\exists_{y{'} <_Y y}Q_Y(y{'})$. Since $\C Y$ is an $\exists$-$\wfs$, we get 
$$\exists_{y \in Y}P_Y(y) :\TOT \exists_{y \in Y}\big(P((x_0, y)) \vee \exists_{x <_X x_0}\exists_{z \in Y}Q((x, z))\big),$$
and hence
$$\exists_{y \in Y}P((x_0, y)) \vee \exists_{x <_X x_0}\exists_{y \in Y}Q((x, y)).$$
If $\exists_{y \in Y}P((x_0, y))$, then $\exists_{(x, y) \in X \times Y}P((x,y))$ holds. If $\exists_{x <_X x_0}\exists_{y \in Y}Q((x, y))$, we use $\exists \PWF_X$ as follows. Let the following extensional formulas on $X$:
$$Q_X(x) :\TOT \exists_{y \in Y}Q((x, y)),$$
$$P_X(x) :\TOT \exists_{y \in Y}P((x, y)).$$
We show that
$$Q_X(x) \stackrel{\mathsmaller{\mathsmaller{\mathsmaller{<_{X}}}}} \Rightarrow P_X(x) :\TOT \forall_{x \in X}\big(Q_X(x) \To \big[P_X(x) \vee \exists_{x{'} <_X x}Q_X(x{'})\big]\big).$$
I.e., if $x \in X$, we show that 
$$\exists_{y \in Y}Q((x, y)) \To \big[\exists_{y \in Y}P((x, y)) \vee \exists_{x{'} <_X x}\exists_{y \in Y}Q((x{'}, y))\big].$$
But what we showed in the first part of our proof was the implication
$$\exists_{y \in Y}Q((x_0, y)) \To \big[\exists_{y \in Y}P((x_0, y)) \vee \exists_{x{'} <_X x_0}\exists_{y \in Y}Q((x{'}, y))\big].$$
Since $x_0$ is arbitrary, the required implication follows in the same way. As $\exists_{x \in X}Q_X(x)$ holds by our intitial hypothesis on $Q((x,y))$, by $\exists \PWF_X$ we get $\exists_{x \in X}P_X(x)$, and hence $\exists_{(x, y) \in X \times Y}P((x,y))$.\\
Cases (i)-(iv) follow in a straightforward manner.
\end{proof}

Notice that the projections on $\C X \times_{\mathsmaller{\lex}} \C Y$ are not in $\StrExtFunRel$. 
%We do not see how one can prove Proposition~\ref{prp: lex}(i) starting with $\exists \PWF_X$. 
%In~\cite{RS93}, p.~149,  $\forall \PWF_Y$ is also used first.
It is also immediate to see that the converse to Proposition~\ref{prp: lex}(iv) does not hold, in general.

Proposition~\ref{prp: basic2}(iv) is generalised to the exterior union, or the Sigma-set of a family of $\exists$-$\wfs$ over an index set which is also an $\exists$-$\wfs$. We include both proofs, because they are instructive. First we give the fundamental definition of an indexed family of sets in $(\SetIneq, \StrExtFun)$. A family of sets  indexed by some set $(I, =_I)$ is an
%non-dependent
assignment routine $\chi_0 : I \sto \D V_0$ that 
behaves like a function, that is if $i =_I j$, then $\chi_0(i) =_{\D V_0} \chi_0 (j)$.
A more explicit definition, which is due to Richman, is included in~\cite{BB85},
p.~78 (Problem 2), which is made precise in~\cite{Pe20} by highlighting 
the role of dependent assignment routines in its formulation. In accordance to the second
attitude described in the Introduction, 
this is a proof-relevant definition revealing the witnesses of the 
equality $\chi_0(i) =_{\D V_0} \chi_0 (j)$. In the following definition $\D V_0^{\neq, <}$ is the universe of sets with an extensional inequality and relation, and $\D F^{\neq, <}$ is the set of functions in $\StrExtFunRel$ from $\C X$ to $\C Y$ in $\SetExtIneqRel$. For the notion of a (non-dependent, or dependent) assignment routine, we refer to~\cite{Pe20}.

\begin{definition}\label{def: famofsets}
	If $\C I := (I, =_I, \neq_I, <_I)$ is in $\SetExtIneqRel$, let the \textit{diagonal} $D(I) := \{(i,j) \in I \times I \mid i =_I j\}$ of $I$.
	A \textit{family of sets} in $(\SetExtIneqRel, \StrExtFunRel)$ indexed by $\C I$
	is a pair $\C X := (\chi_0, \chi_1)$, where
	$\chi_0 \colon I \sto \D V_0^{\neq, <}$ and
	$$\C X (i) := \big(\chi_0(i), =_{\chi_0(i)}, \neq_{\chi_0(i)}, <_{\chi_0(i)}\big),$$
	for every $i \in I$,
	and\index{$\chi_1$} $\chi_1$, a
	\textit{modulus of function-likeness for}\index{modulus of function-likeness} $\chi_0$, is 
	a dependent operation
	\[ \chi_1 \colon \bigcurlywedge_{(i, j) \in D(I)}\D F^{\neq, <}\big(\chi_0(i), \chi_0(j)\big), 
	\ \ \ \chi_1(i, j) =: \chi_{ij} \colon \chi_0(i) \to \chi_0(j), \ \ \ (i, j) \in D(I), \]
	such that the \textit{transport maps}\index{transport map of a family of sets} $\chi_{ij}$
	\index{transport map of a family of sets}\index{$\chi_{ij}$}
	of $\boldmath \chi$ satisfy the following conditions:\\[1mm]
	\normalfont (a) 
	\itshape For every $i \in I$, we have that $\chi_{ii} = \id_{\chi_0(i)}$.\\[1mm]
	\normalfont (b) 
	\itshape If $i =_I j$ and $j =_I k$, the following triangle commutes
	\begin{center}
		\begin{tikzpicture}
			
			\node (E) at (0,0) {$\chi_0(j)$};
			\node[right=of E] (F) {$\chi_0(k).$};
			%\node[above=of F] (A) {$B$};
			\node [above=of E] (D) {$\chi_0(i)$};
			%\node [below=of D] (G) {$V$};

			\draw[->] (E)--(F) node [midway,below] {$\chi_{jk}$};
			%\draw[->] (D)--(A) node [midway,above] {$f $};
			\draw[->] (D)--(E) node [midway,left] {$\chi_{ij}$};
			\draw[->] (D)--(F) node [midway,right] {$\ \chi_{ik}$};
			%\draw[->] (D)--(E) node [midway,left] {$f$};

		\end{tikzpicture}
	\end{center}
	%$I$ is the \textit{index-set}\index{index-set} of the family $\Lambda$.
	%$i =_I j$, we call $\lambda_{ij}$\index{$\lambda_{ij}$} the 
	%\textit{transport map}\index{transport map} from $\lambda_0 (i)$ to $\chi_0 (j)$. 
	If $\C X, \C Y$ are in $\SetExtIneqRel$, the \textit{constant} $\C I$-\textit{family of sets}
	$\C X$is the pair 
	$(\chi_0^X, \chi_1^X)$, where $\chi_0 (i) := X$, for every 
	$i \in I$, and $\chi_1 (i, j) := \id_X$, for every $(i, j) \in D(I)$. 
	% If $\B 2 := \big(\D 2 := \{0, 1\},
	% =_{\D 2}, \neq_{\D 2}\big)$,
	%equipped with the obvious equality and inequality,
	The $\D 2$-family of $\C X$ and $\C Y$ in $\SetExtIneqRel$ is defined by $\chi_0 (0) := \C X$, 
	$\chi_0 (1) := \C Y$, $\chi_{00} := \id_X$ and $\chi_{11} := \id_Y$.
\end{definition}

If $i =_I j$, then $(\chi_{ij}, \chi_{ji}) \colon \chi_0(i) =_{\D V_0^{\neq, <}} \chi_0(j)$.
Next we describe the Sigma-set (or the exterior union, or the disjoint union) of a given family of sets 
in $(\SetExtIneqRel, \StrExtFunRel)$.

\begin{proposition}\label{prp: sigmaset}
	Let $\C X := (\chi_0, \chi_1)$ be an $\C I$-family of sets in $(\SetExtIneqRel, \StrExtFunRel)$. Its Sigma-set is the structure 
	$$\sum_{i \in I}\C X_i := \bigg(\sum_{i \in I}\chi_0 (i), =_{\mathsmaller{\sum_{i \in I}\chi_0 (i)}}, 
	\neq_{\mathsmaller{\sum_{i \in I}\chi_0 (i)}}, <_{\mathsmaller{\sum_{i \in I}\chi_0 (i)}}\bigg),$$
	%of $\boldmath \Lambda$
	%and its canonical equality
	where
	\[ w \in \sum_{i \in I}\chi_0 (i) : 
	\TOT \exists_{i \in I}\exists_{x \in \chi_0 (i)}\big(w := (i, x)\big), \]
	\[ (i, x) =_{\mathsmaller{\sum_{i \in I}\chi_0 (i)}} (j, y) : \TOT i =_I j \ \& \ \chi_{ij} (x) 
	=_{\chi_0 (j)} y,\]
	\[(i,x) \neq_{\mathsmaller{\sum_{i \in I}\chi_0 (i)}} (j, y) :\TOT i \neq_I j  \vee  \big(i =_I j \ \& \ 
	\chi_{ij}(x) \neq_{\chi_0(j)} y \big), \]
		\[(i,x) <_{\mathsmaller{\sum_{i \in I}\chi_0 (i)}} (j, y) :\TOT i <_I j  \vee  \big(i =_I j \ \& \ 
	\chi_{ij}(x) <_{\chi_0(j)} y \big). \]
	\normalfont (i)
	\itshape Then $\sum_{i \in I}\C X_i$ is in $\SetExtIneqRel$ and its first projection $\pr_1^{\C X} \colon \sum_{i \in I}\chi_0 (i) \sto I,$ defined by the rule\footnote{The global projection operations $\prb_1$ and $\prb_2$ are primitive operations in $\BST$.}
	$\pr_1^{\C X} (i, x) : = \prb_1 (i, x) := i$, is in $\StrExtFun$, but not in $\StrExtFunRel$. \\[1mm]
	\normalfont (ii)
\itshape If $\C I$ and every $\C X_i$ are strong $($dichotomous$)$, then $\sum_{i \in I}\C X_i$ is strong $($dichotomous$)$. \\[1mm]
	\normalfont (iii)
\itshape If $\C I$ and every $\C X_i$ are in $\exists \WFSet$, then $\sum_{i \in I}\C X_i$ is in $\exists \WFSet$.
\end{proposition}

\begin{proof}
(i) We show that $<_{\mathsmaller{\sum_{i \in I}\chi_0 (i)}}$ is extensional. If
	\[ (i, x) =_{\mathsmaller{\sum_{i \in I}\chi_0 (i)}} (i{'}, x{'}) : \TOT i =_I i{'} \ \& \ \chi_{ii{'}} (x) 
=_{\chi_0 (i{'})} x{'},\]
	\[ (j, y) =_{\mathsmaller{\sum_{i \in I}\chi_0 (i)}} (j{'}, y{'}) : \TOT j =_I j{'} \ \& \ \chi_{jj{'}} (y) 
=_{\chi_0 (j{'})} y{'},\]
\[(i,x) <_{\mathsmaller{\sum_{i \in I}\chi_0 (i)}} (j, y) :\TOT i <_I j  \vee  \big(i =_I j \ \& \ 
\chi_{ij}(x) <_{\chi_0(j)} y \big),\]
then we show that
\[(i{'},x{'}) <_{\mathsmaller{\sum_{i \in I}\chi_0 (i)}} (j{'}, y{'}) :\TOT i{'} <_I j{'} \vee  \big(i{'} =_I j{'} \ \& \ 
\chi_{i{'}j{'}}(x{'}) <_{\chi_0(j{'})} y{'} \big).\]
If $i <_I j$, then we get $i{'} <_I j{'}$ by the extensionality of $<_I$. If $i =_I j \ \& \ 
\chi_{ij}(x) <_{\chi_0(j)} y$, then we get trivially that $i{'} =_I j{'}$. To show $\chi_{i{'}j{'}}(x{'}) <_{\chi_0(j{'})} y{'}$, we 
first
observe that by Definition~\ref{def: famofsets} we have that
$$\chi_{i{'}j{'}}(x{'}) =_{\chi_0 (j{'})} \chi_{i{'}j{'}}\big(\chi_{ii{'}} (x)\big) =_{\chi_0 (j{'})} \chi_{ij{'}} (x).$$
Since the transport maps preserve the corresponding relations, we have that
$$\chi_{ij}(x) <_{\chi_0(j)} y \To \chi_{jj{'}}\big(\chi_{ij}(x)\big) <_{\chi_0(j{'})} \chi_{jj{'}}(y) \TOT 
\chi_{ij{'}} (x) <_{\chi_0(j{'})} y \TOT \chi_{i{'}j{'}}(x{'}) <_{\chi_0(j{'})} y{'}.$$
The extensionality of $\neq_{\mathsmaller{\sum_{i \in I}\chi_0 (i)}}$ is shown similarly. The assignment routine $\pr_1^{\C X}$ is trivially a strongly extensional function, but it does not preserve, in general, the corresponding relations. If $(i,x) <_{\mathsmaller{\sum_{i \in I}\chi_0 (i)}} (j, y)$ because $i =_I j \ \& \ 
\chi_{ij}(x) <_{\chi_0(j)} y$, then by the extensionality of $<_I$ and Proposition~\ref{prp: basic1}(i) we get
$$\pr_1^{\C X}((i, x)) <_I \pr_1^{\C X}((y, w)) :\TOT i <_I j \To i <_I i \To \bot_{\Nat}.$$
(ii) We only show that $\sum_{i \in I}\C X_i$ is dichotomous. Let $(i,x) \neq_{\mathsmaller{\sum_{i \in I}\chi_0 (i)}} (j, y)$. If $i \neq_i j$, then, since $\C I$ is dichotomous, we get $i <_I j$ or $j <_I i$, and hence $(i,x) <_{\mathsmaller{\sum_{i \in I}\chi_0 (i)}} (j, y)$ or $(j,y) <_{\mathsmaller{\sum_{i \in I}\chi_0 (i)}} (i, x)$. If $i =_I j$ and $\chi_{ij}(x) \neq_{\chi_0(j)} y$, then, since $\C X_j$ is dichotomous, we get $\chi_{ij}(x) <_{\chi_0(j)} y$, and hence $(i,x) <_{\mathsmaller{\sum_{i \in I}\chi_0 (i)}} (j, y)$ or $y <_{\chi_0(j)} \chi_{ij}(x)$, and hence $(j,y) <_{\mathsmaller{\sum_{i \in I}\chi_0 (i)}} (i, x)$, since.
$$y <_{\chi_0(j)} \chi_{ij}(x) \To \chi_{ji}(y) <_{\chi_0(i)}  \chi_{ji}\big(\chi_{ij}(x)\big) \TOT \chi_{ji}(y) <_{\chi_0(i)} \chi_{ii}(x) \TOT \chi_{ji}(y) <_{\chi_0(i)} x.$$
(iii) Let extensional formulas $Q(w)$ and $P(w)$ on $\sum_{i \in I}\chi_0 (i)$, such that $Q(w) \stackrel{\mathsmaller{\mathsmaller{\mathsmaller{<_{\sum_{i \in I}\chi_0(i)}}}}} \Longrightarrow P(w)$. Let also $(i_0, x_0) \in \sum_{i \in I}\chi_0(i)$, such that $Q((i, x_0))$. Let the extensional formulas $Q_{i_0}(x) :\TOT Q((i_0, x))$ and 
$$P_{i_0}(x) :\TOT P((i_0, x)) \vee \exists_{i <_I i_{0}}\exists_{x{'} \in \chi_0(i)}Q((i, x{'}))$$
 on $\chi_0(i_0)$. We show that $Q_{i_0}(x) \stackrel{\mathsmaller{\mathsmaller{\mathsmaller{<_{\chi_0(i_0)}}}}} \Longrightarrow P_{i_0}(x)$. If $x \in \chi_0(i_0)$, such that $Q((i_0, x))$, then by our hypothesis $Q(w) \stackrel{\mathsmaller{\mathsmaller{\mathsmaller{<_{\sum_{i \in I}\chi_0(i)}}}}} \Longrightarrow P(w)$ we get $P((i_0, x))$ or there is $(i, x{'}) \in \sum_{i \in I}\chi_0(i)$ with $(i, x{'}) <_{\mathsmaller{\sum_{i \in I}\chi_0(i)}} (i_0, x)$ and $Q((i, x{'}))$. In the latter case, either $i <_I i_0$ with $Q((i, x{'}))$ or $i =_I i_0$ and $\chi_{ii_0}(x{'}) <_{\mathsmaller{\chi_0(i_0)}} x$ with $Q((i, x{'}))$. The first two cases trivially imply $P_{i_0}(x)$. By the extensionality of $Q(w)$ we have that
 $$\big[(i, x{'}) =_{\mathsmaller{\sum_{i \in I}\chi_0(i)}} \big(i_0, \chi_{ii_0}(x{'}\big) \ \& \ Q((i, x{'}))\big] \To Q\big(\big(i_0, \chi_{ii_0}(x{'})\big)\big),$$
and hence $Q_{i_0}(\chi_{ii_0}(x{'}))$. Since $\C X_{i_0}$ is an $\exists$-$\wfs$, we get that 
$$\exists_{x \in \chi_0(i)}\big(P((i_0, x)) \vee \exists_{i <_I i_{0}}\exists_{x{'} \in \chi_0(i)}Q((i, x{'}))\big),$$
hence $\exists_{x \in \chi_0(i)}P((i, x_0))$, which implies trivially that $\exists_{w \in \sum_{i \in I}\chi_0(i)}Q(w)$, or 
$\exists_{i <_I i_{0}}\exists_{x{'} \in \chi_0(i)}Q((i, x{'})).$
In the latter case we define the following extensional formulas on $I$:
$$Q_I(i) :\TOT \exists_{x \in \chi_0(i)}Q((i, x)), \ \ \ \ P_I(i) :\TOT \exists_{x \in \chi_0(i)}P((i, x)).$$
We show that 
$$Q_I(i) \stackrel{\mathsmaller{\mathsmaller{\mathsmaller{<_{I}}}}} \Rightarrow P_I(i) :\TOT \forall_{i \in I}\big(Q_I(i) \To \big[P_I(i) \vee \exists_{i{'} <_I i}Q_I(i{'})\big]\big).$$
If we fix $i \in I$ with $Q_I(i)$, then by repeating the previous proof of $Q_{i_0}(x) \stackrel{\mathsmaller{\mathsmaller{\mathsmaller{<_{\chi_0(i_0)}}}}} \Longrightarrow P_{i_0}(x)$ in the case of $\C X_i$, we get exactly the required disjumction $P_I(i) \ \vee \ \exists_{i{'} <_I i}Q_I(i{'})$. Since 
$\exists_{i \in I}Q_I(i)$ by the conclusion of the last third case, we get by $\exists \PWF_I$ that $\exists_{i \in I}P_I(i)$, hence $\exists_{w \in \sum_{i \in I}\chi_0(i)}P(w)$.
\end{proof}

Clearly, the Sigma-set of the $\D 2$-family of $\C X$ and $\C Y$ is their coproduct $\C X + \C Y$, and Proposition~\ref{prp: basic2}(iv) is a special case of Proposition~\ref{prp: sigmaset}. By Proposition~\ref{prp: sigmaset}(i) in the category $(\exists \WFSet, \StrExtFunRel)$ the Sigma-sets of families in it are not Sigma-objects in the sense of Pitts~\cite{Pi00} (see also~\cite{Pe23}). Notice that the fact that $\pr_1^{\C X}$ is not in $\StrExtFunRel$ explains why we cannot use Proposition~\ref{prp: basic2}(i) in order to show that the Sigma-set of a family in $\exists \WFSet$ is also in $\exists \WFSet$.

%
%
%\section{Concluding comments}
%\label{sec: concl}
%
%
%
%$\LNP$ is classical well-foundedness of $\Nat$
%
%We could have formulated everything for subsets, but we didn't because 
%
%
%
%
%Why we propose to include $\exists  \PWF_{\Nat}$ as a new axiom:
%
%More general notion of a well-founded relation? 
%
%Our concept ``non-inductive''
%
%Is the type-theoretic version of $\exists \PWF_{\Nat}$ provable in $\MLTT$?
%
%
%
%
%
%
%
%
%
%
%
%%Acknowledgements.......PS, TC?, HL?
%
%

\end{document}